\documentclass[leqno,11pt]{amsart}
\usepackage{amssymb, amsmath}
\usepackage{amsthm, amsfonts,mathrsfs}
\def\inte#1{
\displaystyle\mathop{#1\kern0pt}^\circ }

\DeclareMathOperator{\sign}{sign}

\def\eqdefa{\buildrel\hbox{\footnotesize def}\over =}

\def\virgp{\raise 2pt\hbox{,}}
\def\cdotpv{\raise 2pt\hbox{;}}

\def\eqdefa{\buildrel\hbox{\footnotesize def}\over =}

\def\C{\mathop{\mathbb C\kern 0pt}\nolimits}
\def\DD{\mathop{\mathbb D\kern 0pt}\nolimits}
\def\EE{\mathop{{\mathbb E \kern 0pt}}\nolimits}
\def\K{\mathop{\mathbb K\kern 0pt}\nolimits}
\def\N{\mathop{\mathbb N\kern 0pt}\nolimits}
\def\Q{\mathop{\mathbb Q\kern 0pt}\nolimits}
\def\R{\mathop{\mathbb R\kern 0pt}\nolimits}
\def\SS{\mathop{\mathbb S\kern 0pt}\nolimits}
\def\ZZ{\mathop{\mathbb Z\kern 0pt}\nolimits}
\def\TT{\mathop{\mathbb T\kern 0pt}\nolimits}
\def\P{\mathop{\mathbb P\kern 0pt}\nolimits}

\def\dv{\mbox{div}}

\def\dive{\mathop{\rm div}\nolimits}

\newcommand{\beq}{\begin{equation}}
\newcommand{\eeq}{\end{equation}}
\newcommand{\ben}{\begin{eqnarray}}
\newcommand{\een}{\end{eqnarray}}
\newcommand{\beno}{\begin{eqnarray*}}
\newcommand{\eeno}{\end{eqnarray*}}

\newtheorem{defi}{Definition}[section]
\newtheorem{thm}{Theorem}[section]
\newtheorem{lem}{Lemma}[section]
\newtheorem{rmk}{Remark}[section]
\newtheorem{col}{Corollary}[section]
\newtheorem{prop}{Proposition}[section]
\renewcommand{\theequation}{\thesection.\arabic{equation}}

\begin{document}
\title[GWP of axisymmetric MHD]
{Low regularity global well-posedness of axisymmetric MHD equations with vertical dissipation and magnetic diffusion}
\author[H.  Abidi]{Hammadi Abidi}
\address[H.  Abidi]{D\'epartement de Math\'ematiques
Facult\'e des Sciences de Tunis
Universit\'e de Tunis EI Manar
2092
Tunis
Tunisia}\email{hammadi.abidi@fst.utm.tn}
\author[G. Gui]{Guilong Gui}
\address[G. Gui]{School of Mathematics and Computational Science, Xiangtan University, Xiangtan 411105, China}\email{glgui@amss.ac.cn}
\author[X. Ke]{Xueli Ke}
\address[X. Ke]{School of Mathematics and Computational Science, Xiangtan University, Xiangtan 411105, China}\email{kexueli123@126.com}

\setcounter{equation}{0}
\maketitle
\begin{abstract}
Consideration in this paper is the global well-posedness for the 3D axisymmetric MHD equations with only vertical dissipation and vertical magnetic diffusion. The existence of unique low-regularity global solutions of the system with initial data in Lorentz spaces is established by using higher-order energy estimates and real interpolation method.
\end{abstract}

\vskip 0.2cm



\renewcommand{\theequation}{\thesection.\arabic{equation}}
\setcounter{equation}{0}
\section{Introduction}

We consider herein the 3-D incompressible anisotropic MHD equations 
\begin{equation}\label{1.1}
\begin{cases}
	\partial_{t}u+u\cdot\nabla u-(\nu_h\Delta_h+\nu_z\partial_{z}^{2}) u+\nabla\Pi =B\cdot\nabla B \quad\mbox{in}\quad \R^+\times\R^3,\\
	\partial_{t}B+u\cdot\nabla B-(\mu_h\Delta_h +\mu_z\partial_{z}^{2}) B= B\cdot\nabla u,\\
	\dive u=\dive B= 0,\\
	(u,b)|_{t=0}=(u_{0},B_{0}),
\end{cases}
\end{equation}
where the unknowns $ u$, $B$ and $\Pi$ stand for the velocity of the fluid, the magnetic field and the scalar pressure function respectively. The nonnegative constants $\nu_z$ (or $\nu_h$) and $\mu_z$ (or $\mu_h$) are the vertical (or horizontal) kinematic viscosity coefficient and magnetic diffusive coefficient. In \eqref{1.1}, the usual Laplacians in the classical MHD equations are substituted by the anisotropic Laplacians $\nu_h\Delta_h+\nu_z\partial_{z}^{2}$ and $\mu_h\Delta_h +\mu_z\partial_{z}^{2}$.

The classical 3-D incompressible MHD equations 
\begin{equation}\label{classical-eqns-1}
\begin{cases}
	\partial_{t}u+u\cdot\nabla u-\nu\Delta u+\nabla\Pi =B\cdot\nabla B \quad\mbox{in}\quad \R^+\times\R^3,\\
	\partial_{t}B+u\cdot\nabla B-\mu\Delta B= B\cdot\nabla u,\\
	\dive u=\dive B= 0,\\
	(u,b)|_{t=0}=(u_{0},B_{0}),
\end{cases}
\end{equation}
described the motion of electrically conducting fluids (e.g., astrophysics, geophysics, plasma physics and cosmology, see \cite{FC1999, PA2001, ET2000, G2014}). The existence, uniqueness and regularity of system \eqref{classical-eqns-1} has been extensively studied by many mathematicians recently. For the case that $\nu>0$ and $\mu>0$, it is well-known that Duvaut and Lions \cite{DL1972} proved the local existence and uniqueness of solutions to the $d$-D MHD system in the Sobolev space $H^s(\mathbb{R}^d), s\geq d$. They also obtained the global existence of solutions under the condition for small initial data. Later on,  the global well-posedness of the 2-D MHD system with large initial data has been established by Sermange and Teman \cite{ST1983}. However, in the case in which $\nu$ and $\mu$ are all zero (i.e., the ideal MHD equations), the global well-posedness for the ideal MHD system remains a challenging open problem. Consequently, on the one hand, focuses have been on the equilibrium state for the MHD system \eqref{1.1} with partial dissipation \cite{AZ2017,WZ2021}. On the other hand, there are some mathematical papers \cite{HX2005, WLGS2019, ZG2020} devoting to the global existence of the MHD system some partial regularity results and Serrin-type regularity criteria.

Note that when the initial magnetic field $B_0$ is identically zero, the system \eqref{1.1} reduces to the $3$-D anisotropic incompressible Navier-Stokes system
\begin{equation}\label{Aniso-NS-1}
\begin{cases}
\partial_tu+u\cdot\nabla u-(\nu_h\Delta_h+\nu_z\partial_z^2) u+\nabla \Pi=0\quad\mbox{in}\quad \R^+\times\R^3, \\
\dive u=0,\\
u|_{t=0}=u_0,
  \end{cases}
\end{equation}
which has been extensively studied by many mathematicians recently (see \cite{CDGG2000, DI2-2002, Paicu2005, CZ2007, gui2010}). In particular, for the case where  $\nu_h>0$ and $\nu_z=0$ the system \eqref{Aniso-NS-1} has been studied for the first time by Chemin et al. in \cite{CDGG2000}.  More precisely, the authors have proved in \cite{CDGG2000} the local in time existence of the solution when the initial data belongs to the anisotropic Sobolev space $H^{0,{1\over2}+}$. The global well-posedness was proved for initial data which are small enough compared with horizontal viscosity $\nu_h$, and moreover, the uniqueness of the solution was proved for more regular initial data, belonging to the space $H^{0,{3\over2}+}$, which was removed later by  Iftimie \cite{DI2-2002}. The critical case $s=\frac 12$ was studied by Paicu \cite{Paicu2005}, who proved that the system \eqref{Aniso-NS-1} is globally well-posed for small initial data in the anisotropic Besov space  $\dot{B}^{0,{1\over2}}$ compared with $\nu_h$. Furthermore, Chemin and Zhang \cite{CZ2007} obtained a similar result by working in an anisotropic Besov space with negative regularity indexes in the horizontal variable, which allows them to prove the global existence of the solution for horizontal Navier-Stokes equations with highly oscillating initial data in the horizontal variables. We recall that the main idea in the case where $\nu_h>0$ and $\nu_z=0,$ in order to control the vertical derivative was to use the incompressibility condition, namely $\partial_xu^1+\partial_yu^2+\partial_zu^3=0,$ which allows one to obtain a regularizing effect for the vertical component $u^3$ by using the horizontal viscosity.

Contrarily to the above situation, the case $\nu_h>0$ and $\nu_v=0$ is more difficult to study because of the lack of regularity in two horizontal variables. In fact, utilizing a regularizing effect only in the vertical direction seems very difficult to recover any regularization in all variables in the general case. For this reason, many mathematicians turn to studying the well-posedness of some particular cases, axisymmetric flows for example. 

The vector field $u=u(x_1, x_2, z)$ is axisymmetric ("without swirl", i.e. $u^{\theta}\equiv 0$), if and only if, $u^r$ and $u^z$ do not depend on $\theta$ and
\begin{equation*}
u=u^r(r,z)e_r+u^z(r,z)e_z, 
\end{equation*}
where
\begin{equation*}
e_r=(\frac{x_1}{r}, \frac{x_2}{r}, 0), \quad e_\theta=(-\frac{x_2}{r}, \frac{x_1}{r}, 0), \quad e_z=(0, 0, 1),\quad r=\sqrt{x_1^2+x_2^2},\quad \theta=\arctan \frac{x_2}{x_1}.
\end{equation*}
A scalar function is called axisymmetric if it has no dependencies on the angular variable  $\theta$.

Indeed, for axisymmetric solutions, we have ${\mathop{\rm div}}\,u=\partial_ru^r+{u^r\over r}+\partial_zu^z=0.$ In the case without swirl, Ukhovskii and Yudovich \cite{UI1968}  studied the global regularity of weak solutions of the axisymmetric Navier-Stokes equations  applying the global regularity of the vorticity and the global {\it a priori } estimate $\|r^{-1}\omega\|_{L^r}\leq\|r^{-1}\omega_{0}\|_{L^r}$ for $r \in [1, +\infty]$.  Later on, Leonardi et al. \cite{LMNP1999} and Abidi  \cite{A2008} independently weakened the regularity assumption for $u_{0}\in H^2(\mathbb{R}^3)$ and $u_{0}\in H^{\frac{1}{2}}(\mathbb{R}^3)$. Furthermore, Abidi and Paicu \cite{AP2018} improved the regularity assumption to $\omega_{0}\in L^{\frac{3}{2},1}(\mathbb{R}^3)$ and $r^{-1}\omega_{0}\in L^{\frac{3}{2},1}(\mathbb{R}^3)$.  The recent breakthrough is from Elgindi \cite{Elgindi2019} on the singularity formation of the 3D Euler equation without swirl with $C^{1, \alpha}$ initial data for the velocity. Other results of axisymmetric Navier-Stokes equations can be found in \cite{HFZ2017, D2016, ZZ2014}.

Similarly, the axisymmetric "without swirl" MHD system in this paper means that the solution of the system \eqref{classical-eqns-1} has the form
\begin{equation}\label{initial-axi-1}
u(t,x_1, x_2, z)=u^r(t,r,z)e_r+u^z(t,r,z)e_z,\quad  B(t, x_1, x_2, z)=B^\theta(t,r,z)e_\theta.
\end{equation} 
The global well-posedness of the  axisymmetric "without swirl" MHD equations \eqref{classical-eqns-1} (with $\nu>0$ and $\mu=0$) was established by Lei \cite{Lei2015} for the initial data $\left(u_0, B_0\right)\in H^{s}(\mathbb{R}^{3})$, $s\geq2$ and $\frac{B^\theta_0}{r}\in {L^{\infty}}$. Recently, Ai and Li \cite{XL2022} weakened the condition to $(u_0,B_0)\in H^1(\mathbb{R}^{3})\times H^{2}(\mathbb{R}^{3})$ and   $\frac{\omega_0}{r}\in {L^{2}}$.  For regularity criteria for the axisymmetric MHD solutions, one may refer to \cite{Y2018, WG2022, ZJ2022} and the references cited therein.
 
Consider the case that the anisotropic Laplacians $\nu_h\Delta_h+\nu_z\partial_{z}^{2}$ and $\mu_h\Delta_h +\mu_z\partial_{z}^{2}$ have only vertical viscosity and magnetic diffusion, that is, $\nu_h=\mu_h=0$ and $\nu_z>0,\,\mu_z>0$, the system \eqref{1.1} reads as
\begin{equation}\label{1.2}
	\begin{cases}
	\partial_{t}u+u\cdot\nabla u-\nu_z\partial_{z}^{2} u+\nabla\Pi =B\cdot\nabla B \quad\mbox{in}\quad \R^+\times\R^3,\\
	\partial_{t}B+u\cdot\nabla B-\mu_z\partial_{z}^{2}B=B\cdot\nabla u,\\
	\dive u=\dive B= 0,\\
	(u, B)|_{t=0}=(u_{0}, B_{0}),
	\end{cases}
	\end{equation}
and its corresponding axisymmetric "without swirl" MHD system can be reformulated as
\begin{equation}\label{1.3}
\begin{cases}
\partial_{t} u^{r}+(u^{r}\partial_{r}+ u^{z}\partial_{z})u^{r}+\partial_{r} \Pi
-\partial_{z}^{2}u^r=-\frac{(B^{\theta})^{2}}{r},\\
\partial_t u^z+(u^{r}\partial_{r}+ u^{z}\partial_{z})u^{z}+\partial_{z}\Pi-\partial_{z}^{2} u^{z}=0,\\
\partial_t B^{\theta} +(u^{r}\partial_{r} + u^{z}\partial_{z}) B^{\theta}-\partial_{z}^{2} B^{\theta} =\frac{u^{r} B^{\theta}}{r},\\
\partial_{r} u^{r}+\frac{u^{r}}{r}+\partial_{z} u^{z}=0.
\end{cases}
\end{equation}
For the initial data $(u_0, B_0)\in H^{2}(\mathbb{R}^{3})$, and $\frac{B^\theta_0}{r}\in L^2\cap L^{\infty}(\mathbb{R}^{3})$, Wang and Guo \cite{WG2022} established the existence of the unique global axisymmetric solutions to the system \eqref{1.2}.

We remark that in previous works, all well-posedness results established require the initial data with high regularity. From this, the Lipschitz norm of the velocity $u$ is locally integrable with respect to time $t$ in $\mathbb{R}^+$, which ensures the propagation of the regularity of the solution. A natural and important question is whether a corresponding well-posedness result can be obtained for low-regularity initial data. This kind of result may be helpful to understand the possible blow-up mechanism of the solution to the system \eqref{1.2}, and shows that the model is applicable for general data without high regularity.

Our aim is to establish a family of low regularity global unique solutions to the axisymmetric "without swirl" MHD equations \eqref{1.2}. Notice that the vorticity $\nabla\times u=\omega^{\theta}e_{\theta}$ with $\omega^{\theta}\eqdefa \partial_{z}u^{r}-\partial_{r}u^{z}$. Setting $\omega\eqdefa \omega^{\theta}$ and $b\eqdefa B^{\theta}$, we know from \eqref{1.3} that $(\omega, \,b)$ satisfies
\begin{equation}\label{1.4}
\begin{cases}
&\partial_{t} \omega+(u\cdot\nabla)\omega-\partial_{z}^{2}\omega=-\partial_{z}(\frac{b^{2}}{r})+\frac{ u^{r}}{r}\omega,\\
&\partial_t b +(u\cdot\nabla)b-\partial_{z}^{2} b =\frac{u^{r} }{r}\,b,\\
&u=(-\Delta)^{-1} \nabla\times (\omega\,e_{\theta}),
\end{cases}
\end{equation}
where the operator $u\cdot\nabla\eqdefa u^r\partial_r+ u^z\partial_z$. 

Our main result is given as follows.

\begin{thm}\label{thm1.1}
Let  the initial data $(\omega_{0},b_{0})$ satisfy
\begin{equation}\label{initial-exist-1}
\omega_{0},\,r^{-1} \omega_0 \in L^{\frac{3}{2},1},\quad b_{0}\in L^{3, 2},\quad r^{-1} b_{0}\in L^{\frac{3}{2}, 1}\cap L^{3, 2}.
\end{equation}
 Let $u_0$ be an axisymmetric
solenoidal vector-field with vorticity $\omega_0 e_\theta$ which is given by the Biot-Savart law:
$$
u_0(X)={1\over4\pi}\int_{\mathbb{R}^3}\frac{(X-Y)
\times(\omega_0e_{\theta})(Y)}
{\vert X-Y\vert^3}\,dY,
$$
and $B_0$ be an axisymmetric solenoidal vector-field with the form $B_0=b_0\,e_\theta$.
Then, the system \eqref{1.2} has a global in time axisymmetric solution  $(u, B)$ such that the vorticity  $\omega$ and the magnetic field $be_{\theta}$ satisfy
\begin{equation*} 
\begin{split}
&\omega,\,r^{-1}\omega,\,r^{-1}b \in C \bigl(\mathbb{R}_+;\,L^{{3\over 2},1}(\mathbb{R}^3)\bigr),
\quad
\partial_z\omega,\,r^{-1}\partial_z\omega,\,r^{-1}\partial_zb  \in L^2_{loc}\bigl(\mathbb{R}_+;\,L^{{3\over 2},1}(\mathbb{R}^3)\bigr),
\\
& b, \,r^{-1}b\in C\bigl(\mathbb{R}_+;\,L^{3,2}(\mathbb{R}^3)\bigr).
\end{split}
\end{equation*}
Moreover, if, in addition, the initial data $(\omega_{0},b_{0})$ satisfies
\begin{equation}\label{initial-unique-1}
\omega_0\in L^{3, 1},\quad
\partial_r\omega_0\in L^{{3\over2}}, \quad  b_{0},\,r^{-1}b_{0} \in \dot{H}^1,
\end{equation}
then the vorticity  $\omega$ and the magnetic field $be_{\theta}$ also satisfy
\begin{align*}
&\omega\in C\bigl(\mathbb{R}_+;\,L^{3, 1}(\mathbb{R}^3)\bigr),\quad\partial_r\omega\in C\bigl(\mathbb{R}_+;\,L^{3\over 2}(\mathbb{R}^3)\bigr),\quad
\partial_z\partial_r\omega\in L^2_{loc}\bigl(\mathbb{R}_+;\,L^{3\over 2}(\mathbb{R}^3)\bigr),\\
&(\nabla b,\nabla\frac{b}{r})\in C\bigl(\mathbb{R}_+;\,L^{2}(\mathbb{R}^3)\bigr),\quad
(\partial_z\nabla b,\partial_z\nabla\frac{b}{r})\in L^2_{loc}\bigl(\mathbb{R}_+;\,L^2(\mathbb{R}^3)\bigr),
\end{align*}
and the solution is unique.
\end{thm}

\begin{rmk} 
Theorem \ref{thm1.1} coincides with the primary conclusion of the Navier-Stokes equations in \cite{AP2018} if the initial magnetic field $B_0\equiv 0$.  In comparison to the result in \cite{WG2022}, the uniqueness of solutions to the MHD equations \eqref{1.3} with the low-regularity initial data in Theorem \ref{thm1.1} is more challenging due to the lack of the control about the velocity $u$ in $L^1_{loc}(\mathbb{R}^+; Lip)$.
\end{rmk}

The proof of Theorem \ref{thm1.1} is completed in Section \ref{sect-4}. We now present a summary of the principal difficulties we encounter in our analysis as well as a sketch of the key ideas used in our proof.

To obtain the existence and uniqueness of regular solutions of the system \eqref{1.2}, we need to establish some higher-order estimates of the velocity field for all $T>0$. Since the system is degenerate along the horizontal direction, it is necessary to establish some new {\it a priori } estimates that overcome the difficulties caused by the lack of smoothing effects in the horizontal direction. To achieve this, some high-order estimates should be obtained from the system \eqref{1.4} about $b$ and the vorticity $\nabla\times u=\omega e_{\theta}$ with $\omega= \partial_{z}u^{r}-\partial_{r}u^{z}$. As in the study of the 3-D axisymmetric Euler equations, for the global existence of the solution to the system \eqref{1.4}, the point is to control the quantity $\|r^{-1} u^r\|_{L^1_t(L^\infty)}$. Indeed, compared with the case in the 3-D axisymmetric Euler equations, the vertical dissipation provides the bound of $\|r^{-1} u^r\|_{L^1_t(L^\infty)}$ by $\|\partial_z{\omega\over r}\|_{L^1_t(L^{{3\over2},1})}$ according to Proposition \ref{prop2.2}.  
Toward this, we introduce the unknowns $(\Omega,\Gamma):=\left(\frac{\omega}{r},\frac{b}{r}\right) $ satisfying
\begin{equation}\label{1.5}
\begin{cases}
\partial_t \Omega+(u\cdot\nabla)\Omega-\partial_{z}^{2}\Omega=-\partial_{z}(\Gamma^{2}),\\
\partial_t\Gamma+(u\cdot\nabla)\Gamma-\partial_{z}^{2}\Gamma=0.
\end{cases}
\end{equation}	
The energy method applied to \eqref{1.5} may give necessary {\it a priori} estimates for the proof of the global existence of the solution to the system \eqref{1.4} with the initial data \eqref{initial-exist-1}. Nevertheless, it's subtle to get the uniqueness of the solution to \eqref{1.4} since the above estimates is not sufficient to ensure the control of the quantity $\|\nabla\,u\|_{L^1_t(L^\infty)}$. Our strategy for proving the uniqueness lies in estimating the system (see \eqref{differ-eqns-1} in Section \ref{sect-4}) satisfied by the differences between two solutions with the same initial data.
Due to the presence of the vertical dissipations in \eqref{differ-eqns-1}, we need only to bound the quantities $\mathcal{F}_i(t)$ with $i=1,...,5$ in Section \ref{sect-4}. Toward this, we can adopt the energy method to get the bounds of these quantities under the assumptions \eqref{initial-unique-1}.

The rest of the paper is organized as follows. In Section \ref{sect-2}, we recall some  properties of the Lorentz spaces and basic lemmas on axisymmetric functions. Section \ref{sect-3} is devoted to some a priori estimates for the system \eqref{1.2}.  Finally, we present the proof of Theorem \ref{thm1.1}  in Section \ref{sect-4}. 
\medbreak \noindent{\bf Notations:} We shall denote
$\int_{\mathbb{R}^3}\cdot dx=2\pi\int_{0}^{\infty}\int_{\mathbb{R}}\cdot rdrdz $.  For $A\lesssim B$, it means that there exists a universal constant $C$, which may vary from line to line, such that $A\leq CB$. Given a Banach space $X$, we shall use $\left(a|b\right)$ to represent the $L^{2}(\mathbb{R}^{3})$ inner product of $a$ and $b$, and $\|\left(a,b \right) \|_{X}=\|a\|_{X}+\|b \|_{X}$. The notation $C_{p}$  is a positive constant depending on $p$.

\renewcommand{\theequation}{\thesection.\arabic{equation}}
\setcounter{equation}{0}

\section{Preliminaries}\label{sect-2}
\setcounter{equation}{0} \vskip .1in
Before to introduce the definition of the Lorentz space, we begin by recalling the rearrangement  of a function. For a measurable function $f$ we define its non-increasing rearrangement  by
 $f^{*}:\mathbb{R}_+\to \mathbb{R}_+$ by
$$
f^{*}(\lambda)\eqdefa \inf\Big\{s\geq0;\,
\big|\{x|\,|f(x)|>s\}\big|\leq\lambda\Big\}.
$$
\begin{defi} (Lorentz spaces, see \cite{Bergh-1976})\label{def2.1}
Let $f$ be a mesurable function and
$1\leq p,\,q\leq\infty.$
Then $f$ belongs to the Lorentz space $L^{p,q}$ if
\begin{equation*}
\|f\|_{L^{p,q}}\eqdefa
\begin{cases}
            \Big( \int^\infty_0(t^{1\over p}f^*(t))^q{dt\over t}\Big)^{1\over q}<\infty&
\text{if}\quad q<\infty,\\
           \sup_{t>0}t^{1\over p}f^*(t)<\infty &\text
             {if}\quad q=\infty.
        \end{cases}
\end{equation*}
\end{defi}
Alternatively, we can define the Lorentz spaces by the real interpolation (see \cite{Bergh-1976}), as the interpolation between Lebesgue spaces:
$$
L^{p,q}\eqdefa(L^{p_0},L^{p_1})_{(\theta,q)},
$$
with $1\le p_0<p<p_1\le\infty,$ $0<\theta<1$ satisfying
${1\over p}={1-\theta\over p_0}+{\theta\over p_1}$ and $1\leq q\leq\infty,$ also $f\in L^{p,q}$ if the following quantity
$$
\|f\|_{L^{p,q}}\eqdefa
\Big(\int_0^\infty\big(t^{-\theta}K(t,f)\big)^q
{dt\over t}\Big)^{1\over q}
$$
is finite with
$$
K(f,t)\eqdefa\displaystyle\inf_{f=f_0+f_1}
\big\{\|f_0\|_{L^{p_0}}+t\|f_1\|_{L^{p_1}}\;\,\big|
\;f_0\in L^{p_0},\,f_1\in L^{p_1}\big\}.
$$
The Lorentz spaces verify the following properties
(see  \cite{ON1963} for more details)~:
\begin{prop}\label{Neil}
Let $f\in L^{p_1,q_1},$ $g\in L^{p_2,q_2}$ and
$1\leq p,q,p_j,q_j\leq\infty$ for $j=1,\,2.$
\begin{itemize}
\item If $1<p<\infty$ and $1\le q\le\infty,$ then
$$
\|fg\|_{L^{p,q}}
\lesssim
\|f\|_{L^{p,q}}\|g\|_{L^{\infty}}.
$$

\item
If ${1\over p}={1\over p_1}+{1\over p_2}$ and
${1\over q}={1\over q_1}+{1\over q_2},$ then
$$
\|fg\|_{L^{p,q}}
\lesssim
\|f\|_{L^{p_1,q_1}}\|g\|_{L^{p_2,q_2}}.
$$

\item If $1<p<\infty$ and $1\le q\le\infty,$ then
$$
\|f\ast g\|_{L^{p,q}}
\lesssim
\|f\|_{L^{p,q}}\|g\|_{L^1}.
$$

\item
If $1<p,\,p_1,\,p_2<\infty,$ ${1\over p}+1={1\over p_1}+{1\over p_2}$ and
${1\over q}={1\over q_1}+{1\over q_2},$ then
$$
\|f\ast g\|_{L^{p,q}}
\lesssim
\|f\|_{L^{p_1,q_1}}\|g\|_{L^{p_2,q_2}},
$$
for $p=\infty,$ and ${1\over q_1}+{1\over q_2}=1,$ then
$$
\|f\ast g\|_{L^\infty}
\lesssim
\|f\|_{L^{p_1,q_1}}\|g\|_{L^{p_2,q_2}}.
$$

\item
For $1\leq p\leq\infty$ and $1\leq q_1\leq q_2\leq\infty,$ we have
$$
L^{p,q_1}\hookrightarrow L^{p,q_2}
\hspace{1cm}\mbox{and}\hspace{1cm}L^{p,p}=L^p.
$$
\end{itemize}
\end{prop}
Let us recall also the interpolation inequality (see \cite{MICHAE1974}) which allows us to obtain some embeddings of spaces.
\begin{lem}\label{lem2.1}
Let $p_0,\,p_1,\,p,\;q$ in $[1,+\infty]$ and $0<\theta<1.$
\begin{itemize}
\item If $q\le p,$ then
$$
\bigl[L^p(L^{p_0}),L^p(L^{p_1})\bigr]_{(\theta,q)}
\hookrightarrow
L^p\bigl([L^{p_0},L^{p_1}]_{(\theta,q)}\bigr).
$$
\item If $p\le q,$ then
$$
L^p\bigl([L^{p_0},L^{p_1}]_{(\theta,q)}\bigr)
\hookrightarrow
\bigl[L^p(L^{p_0}),L^p(L^{p_1})\bigr]_{(\theta,q)}.
$$

\end{itemize}
\end{lem}
Recall also the definition of  Lebesgue anisotropic spaces. Denote
 the space $L^p_v(\mathbb{R};\,L^q(\mathbb{R}^2))$ by $L^p_v (L^q_h)$ with the norm
$$
\|f\|_{L^p_v (L^q_h)}
\eqdefa\big(\int_{\mathbb{R}}\big(\int_{\mathbb{R}^2}
|f(x,y,z)|^q\,dxdy\big)^{\frac pq}\,dz\big)^{\frac1p}.
$$
Similarly, we denote by $L^q_h(L^p_v)$ the space
$L^q(\mathbb{R}^2;\,L^p(\mathbb{R})),$ with the norm
$$
\|f\|_{L^q_h(L^p_v)}
\eqdefa\big(\int_{\mathbb{R}^2}\big(\int_{\mathbb{R}}
|f(x,y,z)|^p\,dz\big)^{\frac qp}\,dxdy\big)^{\frac1q}.
$$

\begin{lem}\label{lem2.2} (See Lemma 3.1 in \cite{AP2018})
Let $1\leq p\leq2$ and $f\in L^p(\mathbb{R}^n)$ such that
$\partial_i\vert f\vert^{\frac{p}{2}}\in L^2(\mathbb{R}^n).$ Then
\begin{equation}\label{psmall-1}
\Vert \partial_if\Vert_{L^p}
\lesssim
\|\partial_i|f|^{\frac{p}{2}}\|_{L^2}
\| f\|_{L^p}^{\frac{2-p}{2}}.
\end{equation}
\end{lem}
Thanks to Proposition 3.1 in \cite{AP2018}, we readily get the following proposition (up to a slight modification).
\begin{prop}\label{prop2.2}
Let $u$ and $b$ be axisymmetric solenoidal vector-field and scalar function respectively with vorticity $\omega=\omega^\theta e_\theta$, which solves the system \eqref{1.4}. Let $(p,q,\lambda)\in[1,\infty]^3,$ then we have
$$
u^r=\omega^\theta=b=0
\quad \mbox{on the axis}\quad r=0,
$$
and the following inequalities :
\vspace*{0,2cm}

\begin{itemize}
\item If ${3\over2}\leq p<\infty$ such that
${1\over q}={1\over3}+{1\over p},$ then
\begin{equation*}
\begin{split}
&\|u\|_{L^{p,\lambda}}\lesssim\|\omega\|_{L^{q,\lambda}},\quad
\|{u^r\over r}\|_{L^{p,\lambda}}\lesssim\|{\omega\over r}\|_{L^{q,\lambda}},\quad
\|\partial_zu^r\|_{L^{p,\lambda}}\lesssim\|\partial_z\omega\|_{L^{q,\lambda}},\\
&\|\partial_zu^z\|_{L^{p,\lambda}}\lesssim\|\partial_z\omega\|_{L^{q,\lambda}}, \quad
\|\partial_zu^z\|_{L^{p,\lambda}} +\|\partial_ru^z\|_{L^{p,\lambda}}\lesssim\|\partial_r\omega\|_{L^{q,\lambda}}+\|{\omega\over r}\|_{L^{q,\lambda}}.
\end{split}
\end{equation*}
\item If $3\leq p<\infty$ such that
${1\over q}={2\over3}+{1\over p},$ then
\begin{equation*}
\begin{split}
&\|u^r\|_{L^{p,\lambda}} \lesssim \|\partial_z\omega\|_{L^{q,\lambda}},\quad \|{u^r\over r}\|_{L^{p,\lambda}}\lesssim\|\partial_z{\omega\over r}\|_{L^{q,\lambda}},\quad\|u^z\|_{L^{p,\lambda}}\lesssim\|\partial_r\omega\|_{L^{q,\lambda}}+\|{\omega\over r}\|_{L^{q,\lambda}},\\
&\|\partial_zu^z\|_{L^{p,\lambda}}\lesssim\|\partial_z\partial_r\omega\|_{L^{q,\lambda}}+\|\partial_z{\omega\over r}\|_{L^{q,\lambda}},\quad \|\partial_ru^r\|_{L^{p,\lambda}}
\lesssim
\|\partial_z\partial_r\omega\|_{L^{q,\lambda}}
+\|\partial_z{\omega\over r}\|_{L^{q,\lambda}}.
\end{split}
\end{equation*}
\item In the limiting case $p=\infty$, there hold
\begin{equation*}
\begin{split}
&\|u\|_{L^\infty} \lesssim \|\omega\|_{L^{3,1}},
\quad \|u^r\|_{L^{\infty}} \lesssim \|\partial_z\omega\|_{L^{{3\over2},1}},
\quad \|{u^r\over r}\|_{L^{\infty}} \lesssim \|\partial_z{\omega\over r}\|_{L^{{3\over2},1}},\\
&\|u^z\|_{L^{\infty}}\lesssim\|\partial_r\omega\|_{L^{{3\over2},1}}+\|{\omega\over r}\|_{L^{{3\over2},1}},
\quad \|\partial_zu^z\|_{L^\infty} \lesssim \|\partial_z\partial_r\omega\|_{L^{{3\over2},1}} +\|\partial_z{\omega\over r}\|_{L^{{3\over2},1}},\\
&\|\partial_ru^r\|_{L^\infty}
\lesssim
\|\partial_z\partial_r\omega\|_{L^{{3\over2},1}}
+\|\partial_z{\omega\over r}\|_{L^{{3\over2},1}}.
\end{split}
\end{equation*}
\end{itemize}
\end{prop}
The proposition given below can be found in \cite{MZ2012}, which we will use in the proof of higher order estimates of $(u, b)$.
\begin{prop}\label{prop2.3}(\cite{MZ2012})
Let $u$ be a free divergence axisymmetric vector-field without swirl and $\omega=\nabla\times u$. Then there hold
$$\frac{u^r}{r}=\partial_z\Delta^{-1}(\frac{\omega}{r})-2\frac{\partial_r}{r}\Delta^{-1}\partial_z\Delta^{-1}(\frac{\omega}{r})$$
and 
$$\|\partial_{z}(\frac{u^r}{r})\|_{L^p}\leq C\|\frac{\omega}{r}\|_{L^p},\quad 1<p<\infty.$$
\end{prop}

\renewcommand{\theequation}{\thesection.\arabic{equation}}
\setcounter{equation}{0}

\section{A prior estimates}\label{sect-3}
\setcounter{equation}{0} \vskip .1in
\begin{prop}\label{prop3.1}
Assume that $1<p<\infty$, $(\Omega_{0},\Gamma_{0})\in L^{p}\times L^{2p}$, $u$ and $b$ are regular axisymmetric such that $\dv \,u=0$. Let $\Omega \eqdefa \frac{\omega}{r}\in L_{t}^{\infty}(L^{p})$ and $\Gamma\eqdefa\frac{b}{r}\in L_{t}^{\infty}(L^{p})$ be regular solutions of the system \eqref{1.5}. Then there are
\begin{align}\label{3.1}
\|\Gamma(t)\|_{ L^{p}}+C_p\|\partial_{z}|\Gamma|^{\frac{p}{2}}\|_{L_{t}^{2}(L^{2})}^{\frac{2}{p}}\leq\|\Gamma_{0}\|_{ L^{p}},
\end{align}
and
\begin{align}\label{3.2}
\|\Omega(t)\|_{L^{p}}+C_p\|\partial_{z}|\Omega|^{\frac{p}{2}}\|_{L_{t}^{2}(L^{2})}^{\frac{2}{p}}\le C\bigl(\|\Omega_{0}\|_{L^{p}}+\sqrt{t}\|\Gamma_{0}^2\|_{L^{p}}\bigr).
\end{align}
\end{prop}
\begin{proof}
Let's first control $\Gamma$ in Lebesgue spaces. For $1<p<\infty$, multiplying the second equation of \eqref{1.5} by $|\Gamma|^{p-1}\sign(\Gamma)$, and then integrating the result equation, we obtain from $\dv\, u=0$ that
\begin{align}\label{3.3}
\frac{1}{p}\frac{\mathrm{d}}{\mathrm{d}t}\|\Gamma\|_{ L^{p}}^{p}+\frac{4(p-1)}{p^{2}}\|\partial_{z}|\Gamma|^{\frac{p}{2}}\|_{L^{2}}^{2}=0,
\end{align}
which yields \eqref{3.1}.

In order to control $\Omega$ in Lebesgue spaces, we will split it into two cases: $1<p\leq2$ and $2\leq p<\infty$.\\

{\bf Case 1 : $1<p\leq2$.}
Multiplying the first equation \eqref{1.5} by $|\Omega|^{p-1}\sign(\Omega)$, and then integrating the result equation, we find
\begin{align}\label{3.5}
\frac{1}{p}\frac{\mathrm{d}}{\mathrm{d}t}\|\Omega\|_{ L^{p}}^{p}+\frac{4(p-1)}{p^{2}}\|\partial_{z}|\Omega|^{\frac{p}{2}}\|_{L^{2}}^{2}\leq\int_{\mathbb{R}^3}-\partial_{z}\Gamma^2|\Omega|^{p-1}\sign(\Omega)\,dx,
\end{align}
which implies
\begin{align}\label{3.5-11}
\frac{1}{p}\frac{\mathrm{d}}{\mathrm{d}t}\|\Omega\|_{ L^{p}}^{p}+\frac{4(p-1)}{p^{2}}\|\partial_{z}|\Omega|^{\frac{p}{2}}\|_{L^{2}}^{2}\leq\|\partial_{z}\Gamma^2\|_{L^p}\|\Omega\|_{L^p}^{p-1}.
\end{align}
Hence, there holds
\begin{align}\label{3.6}
\|\Omega(t)\|_{L^{p}}+C_{p}\|\partial_{z}|\Omega|^{\frac{p}{2}}\|_{L_{t}^{2}(L^{2})}^{\frac{2}{p}}\leq \|\Omega_{0}\|_{L^{p}}+\int_{0}^{t}\|\partial_{z}\Gamma^2\|_{L^{p}}\mathrm{d}\tau.
\end{align}
In order to close the above inequality, we may obtain the equation of $\Gamma^2$ from the $\Gamma$-equation in \eqref{1.5} that
\begin{equation}\label{3.7}
\partial_t\Gamma^2+(u\cdot\nabla)\Gamma^2-\partial_{z}^{2}\Gamma^2=-2(\partial_{z}\Gamma)^2.
\end{equation}	
Similar to the argument in \eqref{3.5}, we achieve
\begin{align}\label{3.8}
\|\Gamma^{2}(t)\|_{ L^{p}}^{p}+C_{p}\|\partial_{z}|\Gamma^2|^{\frac{p}{2}}\|_{L_{t}^{2}(L^{2})}^{2}\leq\|\Gamma^2_{0}\|_{ L^{p}}^{p}.
\end{align}
Thanks to \eqref{psmall-1}, we have
\begin{align*}
\int_{0}^{t}\|\partial_{z}\Gamma^2\|_{L^{p}}^2\,d\tau\leq C\|\partial_{z}|\Gamma^2|^{\frac{p}{2}}\|_{L_{t}^{2}(L^2)}^2\|\Gamma^2\|_{L_t^\infty(L^p)}^{2-p}\leq C\|\Gamma_{0}^2\|_{L^{p}}^2.
\end{align*}
and then, we get, for $1<p \leq 2$, 
\begin{align}\label{3.1-aa}
\|\Gamma^{2}(t)\|_{ L^{p}}+C_{p}\|\partial_{z}\Gamma^2\|_{L_{t}^{2}(L^{p})}\leq C\|\Gamma^2_{0}\|_{ L^{p}},
\end{align}
and
\begin{align}\label{3.9}
\|\partial_{z}\Gamma^2\|_{L^1_t(L^{p})} \leq \sqrt{t}\|\partial_{z}\Gamma^2\|_{L^2_t(L^{p})} \leq C \sqrt{t}\|\Gamma_{0}^2\|_{L^{p}}.
\end{align}
Inserting \eqref{3.9} into \eqref{3.6} implies \eqref{3.2}.

{\bf Case 2: $2\leq p<+\infty$.}
In view of \eqref{3.5}, we have
\begin{align}\label{3.11}
&\frac{\mathrm{d}}{\mathrm{d}t}\|\Omega\|_{ L^{p}}^{p}+\frac{4(p-1)}{p}\|\partial_{z}|\Omega|^{\frac{p}{2}}\|_{L^{2}}^{2}
\leq\frac{2p-2}{p}|\int_{\mathbb{R}^3}\Gamma^2(\partial_{z}|\Omega|^{\frac{p}{2}})|\Omega|^{\frac{p-2}{2}}\mathrm{d}x|,
\end{align}
which leads to
\begin{align}\label{3.11-11}
&\frac{\mathrm{d}}{\mathrm{d}t}\|\Omega\|_{ L^{p}}^{p}+\frac{4(p-1)}{p}\|\partial_{z}|\Omega|^{\frac{p}{2}}\|_{L^{2}}^{2}
\leq C_{p}\|\partial_{z}|\Omega|^{\frac{p}{2}}\|_{L^2}\|\Gamma^2\|_{L^p}\|\Omega\|_{L^p}^{\frac{p-2}{2}}.
\end{align}
Thence, Young's inequality implies
\begin{align}\label{3.11-22}
&\frac{\mathrm{d}}{\mathrm{d}t}\|\Omega\|_{ L^{p}}^{p}+\frac{2(p-1)}{p}\|\partial_{z}|\Omega|^{\frac{p}{2}}\|_{L^{2}}^{2}
\leq C\|\Gamma^2\|_{L^p}^2\|\Omega\|_{L^p}^{p-2}.
\end{align}
Combining \eqref{3.11-22} with \eqref{3.8}, one obtains \eqref{3.2}.

Therefore, we finish the proof of Proposition \ref{prop3.1}.
\end{proof}

\begin{rmk}\label{rmk-1-1}
For $1<p \leq 2$, thanks to \eqref{psmall-1}, we have
\begin{equation*}
  \|\partial_{z}\Gamma\|_{L^{p}} \lesssim \|\partial_{z}|\Gamma|^{\frac{p}{2}}\|_{L^{2}}  \|\Gamma\|_{L^{p}}^{\frac{2-p}{2}},\quad \|\partial_{z}\Omega\|_{L^{p}} \lesssim \|\partial_{z}|\Omega|^{\frac{p}{2}}\|_{L^{2}}  \|\Omega\|_{L^{p}}^{\frac{2-p}{2}},
\end{equation*}
from which, together with \eqref{3.1} and \eqref{3.2}, one has
\begin{equation}\label{1-1-rmk1}
\begin{split}
& \|\Gamma(t)\|_{ L^{p}}+\|\partial_{z} \Gamma\|_{L_{t}^{2}(L^{p})}\leq C\|\Gamma_{0}\|_{ L^{p}},\\
  &\|\Gamma^{2}(t)\|_{ L^{p}}+\|\partial_{z}\Gamma^2\|_{L_{t}^{2}(L^{p})}\leq C\|\Gamma^2_{0}\|_{ L^{p}},\\
   &\|\Omega(t)\|_{L^{p}}+\|\partial_{z}\Omega\|_{L^2_t(L^{p})} \leq C\bigl(\|\Omega_{0}\|_{L^{p}}+\sqrt{t}\|\Gamma_{0}^2\|_{L^{p}}\bigr).
\end{split}
\end{equation}
\end{rmk}

\begin{rmk}\label{rmk-1-2}
We denote  by ${\mathcal{T}}$ and ${\mathcal{S}}$ the following linear operators:
$$
\begin{aligned}
{\mathcal{T}}:\hspace{1cm} &L^p\longrightarrow L^p
\hspace{2cm}
{\mathcal{S}}:\hspace{1cm} L^p\longrightarrow L^2_t(L^p)
\\&
\Omega_0\longmapsto \Omega
\hspace{3,75cm}
\Omega_0\longmapsto \partial_z\Omega,
\end{aligned}
$$
with $\Omega$ solution of the system \eqref{1.5}.
By definition, we know that ${\mathcal{T}}$ and ${\mathcal{S}}$ are linear operators, then thanks to Propositions \ref{Neil} and \ref{prop3.1}, Lemmas \ref{lem2.1} and \ref{lem2.2}, and Remark \ref{rmk-1-1}, we obtain for $1< p\leq 2$, $1\leq q\leq p$,
\begin{equation}\label{est-Lpq-1}
\begin{split}
&\|\Omega(t)\|_{L^{p,q}}
+\Big\|\partial_z\Omega
\Big\|_{L^2_t(L^{p,q})}\leq C\Big(\|\Omega_0\|_{L^{p,q}}+
\sqrt{t}\|\Gamma_{0}\|_{L^{2p,2q}}^2\Big),\\
&\|\Gamma(t)\|_{ L^{p,q}}+\|\partial_{z}\Gamma\|_{L_{t}^{2}(L^{p,q})}\leq C\|\Gamma_{0}\|_{ L^{p,q}},\quad 
\|\Gamma^2(t)\|_{ L^{p,q}}+\|\partial_{z}\Gamma^2\|_{L_{t}^{2}(L^{p,q})}\leq C\|\Gamma_{0}^2\|_{ L^{p,q}}.
\end{split}
\end{equation}
While for $2< p<\infty$, $1\leq q\leq p$, we have
\begin{equation}\label{est-Lpq-2}
\begin{split}
\|\Omega(t)\|_{L^{p,q}}\leq  C\Big(\|\Omega_0\|_{L^{p,q}}+
\sqrt{t}\|\Gamma_{0}\|_{L^{2p,2q}}^2\Big) \quad \mbox{and} \quad \|\Gamma(t)\|_{ L^{p,q}}\leq\|\Gamma_{0}\|_{ L^{p,q}}.
\end{split}
\end{equation}
\end{rmk}
\begin{col}\label{cro3.1}
Assume that  $(\Omega_{0},\Gamma_{0})\in L^{\frac{3}{2},1}\times L^{3, 2}$ and $u$ is a regular axisymmetric vector field such that $div \,u=0$. Let $\Omega\eqdefa\frac{\omega}{r}\in L_{t}^{\infty}(L^{\frac{3}{2},1})$ and $\Gamma\eqdefa\frac{b}{r}\in L_{t}^{\infty}(L^{3,2})$ be a solution of system \eqref{1.5}. Then there hold
\begin{equation}\label{est-Lpq-3}
\begin{split}
&\|\Gamma(t)\|_{ L^{3,2}}+\|\partial_{z}\Gamma^2\|_{L_{t}^{2}(L^{\frac{3}{2},1})}^{\frac{1}{2}}
\leq
\|\Gamma_{0}\|_{ L^{3,2}},\\
&\|\Omega(t)\|_{L^{\frac{3}{2},1}}+\|\partial_{z}\Omega\|_{L_{t}^{2}(L^{\frac{3}{2},1})}\le C\bigl(\|\Omega_{0}\|_{L^{\frac{3}{2},1}}+\sqrt{t}\|\Gamma_{0}\|_{L^{3,2}}^2\bigr).
\end{split}
\end{equation}
In particular,  we have for all $t\geq 0$,
\begin{align}\label{3.17}
\int_{0}^{t}\|\frac{u^r}{r}\|_{L^{\infty}}\mathrm{d}\tau
\lesssim
\int_{0}^{t}\|\partial_{z}\frac{\omega}{r}\|_{L^{\frac{3}{2},1}}\mathrm{d}\tau\le
C\sqrt{t}\bigl(\|\Omega_{0}\|_{L^{\frac{3}{2},1}}+\sqrt{t}\|\Gamma_{0}\|_{L^{3,2}}^2\bigr).
\end{align}
\end{col}
\begin{prop}\label{prop3.2}
Let $1<p<\infty$, $\frac{\omega_0}{r}\in L^{\frac{3}{2},1}, \quad \frac{b_{0}}{r}\in L^{3, 2},\quad \omega_{0},\,b_{0},\, \frac{b_{0}^2}{r}\in L^p$.
Assume that $(\omega,\,b)\in L_{t}^{\infty}(L^{p}) \times L_{t}^{\infty}(L^{p})$ is a regular solution of the equations \eqref{1.4}. Then there hold
\begin{align}\label{3.13}
\Vert\omega(t)\Vert_{L^p}
+\Big\Vert\partial_z\vert\omega\vert^{\frac{p}{2}}
\Big\Vert_{L^{2}_{t}(L^2)}^{\frac{2}{p}}
\leq C\big(\|\omega_{0}\|_{L^{p}}+ \sqrt{t}\,\|\frac{b_{0}^2}{r}\|_{L^p} \big)e^{CA_0(t)},
\end{align}
\begin{align}\label{b2r-est-0}
&  \| \frac{b^2}{r}(t)\|_{L^p}  +\Big\Vert\partial_z\vert\frac{b^2}{r}\vert^{\frac{p}{2}}
\Big\Vert_{L^2_t(L^2)}^{\frac{2}{p}}  \leq C\| \frac{b_0^2}{r}\|_{L^p} e^{CA_0(t)},
\end{align}
and
\begin{align}\label{3.14}
\|b(t)\|_{L^{p}}+\|\partial_{z}|b|^{\frac{p}{2}}\|_{L_{t}^{2}(L^{2})}^{\frac{2}{p}}
\le
C\|b_{0}\|_{L^{p}}e^{CA_{0}(t)},
\end{align}
where
$$
A_0(t)\eqdefa\sqrt{t}\bigl(\|\Omega_{0}\|_{L^{\frac{3}{2},1}}+\sqrt{t}\,\|\Gamma_{0}\|_{L^{3,2}}^2\bigr).
$$
\end{prop}
\begin{proof}
Due to the second equation in \eqref{1.4}, we find
\begin{align}\label{b2r-eqns}
\partial_t (\frac{b^2}{r}) +(u\cdot\nabla) (\frac{b^2}{r})-\partial_{z}^{2}  (\frac{b^2}{r}) =- \frac{2}{r}(\partial_zb)^2+\frac{u^{r}}{r} (\frac{b^2}{r}).
\end{align}
Hence, for any $1< p<+\infty$, one can see 
\begin{align}\label{b2r-eqns-est-1}
&\frac{1}{p}\frac{\mathrm{d}}{\mathrm{d}t}\| \frac{b^2}{r}\|_{L^p}^p
+{4(p-1)\over p^2}
\Big\Vert\partial_z\vert\frac{b^2}{r}\vert^{\frac{p}{2}}
\Big\Vert_{L^2}^2=-\int_{\mathbb{R}^3}\frac{2}{r}(\partial_zb)^2(\frac{b^2}{r})^{p-1}\,\mathrm{d}x+ \int_{\mathbb{R}^3}\frac{u^{r}}{r} (\frac{b^2}{r})^p\,\mathrm{d}x,
\end{align}
which follows
\begin{align}\label{b2r-eqns-est-2}
&\frac{1}{p}\frac{\mathrm{d}}{\mathrm{d}t}\| \frac{b^2}{r}\|_{L^p}^p
+{4(p-1)\over p^2}
\Big\Vert\partial_z\vert\frac{b^2}{r}\vert^{\frac{p}{2}}
\Big\Vert_{L^2}^2\leq  \|\frac{u^{r}}{r}\|_{L^\infty} \| \frac{b^2}{r}\|_{L^p}^p.
\end{align}
Gronwall's inequality along with \eqref{3.17} leads to
\begin{align}\label{b2r-eqns-est-3}
&  \sup_{\tau\in [0, t]}\| \frac{b^2}{r}(\tau)\|_{L^p}^p +\Big\Vert\partial_z\vert\frac{b^2}{r}\vert^{\frac{p}{2}}
\Big\Vert_{L^2_t(L^2)}^2\leq C\| \frac{b_0^2}{r}\|_{L^p}^p\exp\{C \int_0^t\|\frac{u^{r}}{r}(\tau)\|_{L^\infty}\,\mathrm{d}\tau\} \leq \| \frac{b_0^2}{r}\|_{L^p}^pe^{CA_0(t)}.
\end{align}
Similarly, from the $b$ equation of \eqref{1.4}, we have
\begin{equation}\label{b-eqns-est-1}
\sup_{\tau\in [0, t]}\|b(\tau)\|_{L^{p}}+\Big\|\partial_{z}|b|^{\frac{p}{2}}\Big\|_{L^{2}}^{\frac{2}{p}} \leq C \|b_{0}\|_{L^{p}}e^{CA_{0}(t)}.
\end{equation}

{\bf Case 1: $2\leq p<+\infty$.}
Multiplying the vorticity equation in \eqref{1.4} by $\vert\omega\vert^{p-1}{\rm sign}(\omega)$, and then integrating the result equation, we obtain
\begin{equation}\label{3.15}
  \begin{split}
 & \frac{1}{p}\frac{\mathrm{d}}{\mathrm{d}t}\Vert\omega\Vert_{L^p}^p
+{4(p-1)\over p^2}
\Big\Vert\partial_z\vert\omega\vert^{\frac{p}{2}}
\Big\Vert_{L^2}^2=\int_{\mathbb{R}^3}  \frac{ u^{r}}{r}|\omega|^{p}\,\mathrm{d}x-\int_{\mathbb{R}^3} \partial_{z}(\frac{b^{2}}{r})|\omega|^{p-1}{\rm sign(\omega)}\,\mathrm{d}x\\
&=\int_{\mathbb{R}^3}  \frac{ u^{r}}{r}|\omega|^{p}\,\mathrm{d}x+\int_{\mathbb{R}^3}(\frac{b^{2}}{r}) \partial_{z}(|\omega|^{p-1}{\rm sign(\omega)})\,\mathrm{d}x.
  \end{split}
\end{equation}
Hence, using H\"{o}lder's and Young's inequalities, one has
\begin{equation}\label{3.15-11}
  \begin{split}
&\frac{1}{p}\frac{\mathrm{d}}{\mathrm{d}t}\Vert\omega\Vert_{L^p}^p
+{4(p-1)\over p^2}
\Big\Vert\partial_z\vert\omega\vert^{\frac{p}{2}}
\Big\Vert_{L^2}^2\leq\|\frac{u^r}{r}\|_{L^{\infty}}\|\omega\|_{L^{p}}^{p}+(p-1)\Big\|\partial_{z}|\omega|^{\frac{p}{2}}\Big\|_{L^2}[\int_{\mathbb{R}^3}\frac{b^{4}}{r^2} |\omega|^{p-2}\mathrm{d}x]^{\frac{1}{2}}\\
&\leq \eta\Big\|\partial_{z}|\omega|^{\frac{p}{2}}\Big\|_{L^2}^{2}+\|\frac{u^r}{r}\|_{L^{\infty}}\|\omega\|_{L^{p}}^p+C_{\eta}\|\frac{b^2}{r}\|_{L^p}^{2}\| \omega\|_{L^p}^{p-2}
  \end{split}
\end{equation}
for any positive constant $\eta$. Hence, taking $\eta={2(p-1)\over p^2}$, we get
\begin{equation}\label{3.15-22}
  \begin{split}
&\frac{d}{dt}\Vert\omega\Vert_{L^p}^p
+{2(p-1)\over p}
\Big\Vert\partial_z\vert\omega\vert^{\frac{p}{2}}
\Big\Vert_{L^2}^2\leq  C\|\frac{u^r}{r}\|_{L^{\infty}}\|\omega\|_{L^{p}}^p+C \|\frac{b^2}{r}\|_{L^p}^{2}\| \omega\|_{L^p}^{p-2}.
  \end{split}
\end{equation}
Gronwall's inequality implies
\begin{equation}\label{3.15-33}
  \begin{split}
\sup_{\tau\in [0, t]}\|\omega(\tau)\|_{L^p}^p+\Big\Vert\partial_z\vert\omega\vert^{\frac{p}{2}}
\Big\Vert_{L^{2}_{t}(L^2)}^{2}&\leq C(\|\omega_0\|_{L^p}^p+ \int_0^t\|\frac{b^2}{r}\|_{L^p}^{2}\| \omega\|_{L^p}^{p-2}\,d\tau)e^{C\int_0^t\|\frac{u^r}{r}(\tau)\|_{L^{\infty}}\,d\tau}\\
&\leq C(\|\omega_0\|_{L^p}^p+ t\,\|\frac{b^2}{r}\|_{L^\infty_t(L^p)}^{2}\| \omega\|_{L^\infty_t(L^p)}^{p-2} )e^{CA_0(t)},
  \end{split}
\end{equation}
which along with \eqref{b2r-eqns-est-3} implies
\begin{align}\label{3.16}
\Vert\omega(t)\Vert_{L^p}
+\Big\Vert\partial_z\vert\omega\vert^{\frac{p}{2}}
\Big\Vert_{L^{2}_{t}(L^2)}^{\frac{2}{p}}
\leq C\big(\|\omega_{0}\|_{L^{p}}+ \|\frac{b_{0}^2}{r}\|_{L^p}\sqrt{t}\big)e^{CA_0(t)}.
\end{align}

{\bf Case 2: $1< p < 2$.} The energy estimates result in
\begin{equation}\label{3.19}
  \begin{split}
\frac{1}{p}\frac{\mathrm{d}}{\mathrm{d}t}\Vert\omega\Vert_{L^p}^p
+{4(p-1)\over p^2}
\Big\Vert\partial_z\vert\omega\vert^{\frac{p}{2}}
\Big\Vert_{L^2}^2&=\int_{\mathbb{R}^3}  \frac{ u^{r}}{r}|\omega|^{p}\,dx-\int_{\mathbb{R}^3} \partial_{z}(\frac{b^{2}}{r})|\omega|^{p-1}{\rm sign(\omega)}\,dx\\
&\leq\|\frac{u^r}{r}\|_{L^{\infty}}\|\omega\|_{L^{p}}^{p}+\|\partial_{z}(\frac{b^{2}}{r})\|_{L^p}\|\omega\|_{L^{p}}^{p-1}.
  \end{split}
\end{equation}
Hence, by virtue to \eqref{psmall-1}, we find
\begin{equation}\label{3.19-11}
 \begin{split}
 &\frac{1}{p}\frac{\mathrm{d}}{\mathrm{d}t}\Vert\omega\Vert_{L^p}^p
+{4(p-1)\over p^2}
\Big\Vert\partial_z\vert\omega\vert^{\frac{p}{2}}
\Big\Vert_{L^2}^2\leq\|\frac{u^r}{r}\|_{L^{\infty}}\|\omega\|_{L^{p}}^{p}+C\|\partial_{z}(\frac{b^{2}}{r})^{\frac{p}{2}}\|_{L^2}\Vert \frac{b^{2}}{r}\Vert_{L^p}^{\frac{2-p}{2}}\|\omega\|_{L^{p}}^{p-1}.
  \end{split}
\end{equation}
Thanks to Gronwall's inequality, we deduce
\begin{equation}\label{3.20}
  \begin{split}
\sup_{\tau\in [0, t]}\Vert\omega(\tau)\Vert_{L^p}
+\Big\Vert\partial_z\vert\omega\vert^{\frac{p}{2}}
\Big\Vert_{L^{2}_{t}(L^2)}^{\frac{2}{p}}
\le C\bigl(\|\omega_{0}\|_{L^{p}}+\int_0^t\|\partial_{z}(\frac{b^{2}}{r})^{\frac{p}{2}}\|_{L^2}\Vert \frac{b^{2}}{r}\Vert_{L^p}^{\frac{2-p}{2}}\,\mathrm{d}\tau\bigr)
e^{CA_{0}(t)},
  \end{split}
\end{equation}
which follows
\begin{equation}\label{3.20-11}
  \begin{split}
&\sup_{\tau\in [0, t]}\Vert\omega(\tau)\Vert_{L^p}
+\Big\Vert\partial_z\vert\omega\vert^{\frac{p}{2}}
\Big\Vert_{L^{2}_{t}(L^2)}^{\frac{2}{p}}
\leq C\bigl(\|\omega_{0}\|_{L^{p}}+\sqrt{t}\|\partial_{z}(\frac{b^{2}}{r})^{\frac{p}{2}}\|_{L^2_t(L^2)}\Vert \frac{b^{2}}{r}\Vert_{L^\infty_t(L^p)}^{\frac{2-p}{2}}\bigr)
e^{CA_{0}(t)}.
  \end{split}
\end{equation}
Therefore, due to \eqref{b2r-eqns-est-3}, we find
\begin{equation}\label{3.20-22}
  \begin{split}
\sup_{\tau\in [0, t]}\Vert\omega(\tau)\Vert_{L^p}
+\Big\Vert\partial_z\vert\omega\vert^{\frac{p}{2}}
\Big\Vert_{L^{2}_{t}(L^2)}^{\frac{2}{p}}
\leq C\bigl(\|\omega_{0}\|_{L^{p}}+\sqrt{t}\,\| \frac{b_0^{2}}{r}\|_{L^p} \bigr)
e^{CA_{0}(t)}.
  \end{split}
\end{equation}
This completes the proof of the proposition.
\end{proof}

Thanks to Proposition \ref{prop3.2} and Lemma \ref{lem2.1}, we have the following results.
\begin{col}\label{cro3.2}
Let  the initial data $(\omega_{0},b_{0})$ satisfy
\begin{align*}
\omega_{0}\in L^{p,q},\quad\frac{\omega_0}{r}\in L^{\frac{3}{2},1},\quad b_{0}\in L^{3,\, 2} \cap L^{p,q},\quad {\rm and}\quad \frac{b_{0}}{r}\in L^{2p,2q}\cap L^{3,\, 2}.
\end{align*}
Assume that $1<p<\infty$, $1 \leq q \leq p$,  $\omega\in L_{t}^{\infty}(L^{p,q})$ and $b\in L_{t}^{\infty}(L^{p,q})$ be a solution of equation \eqref{1.4}. Then there are
\begin{itemize}
  \item if $1<p\leq 2$, $1 \leq q \leq p$, then
  \begin{itemize}
  \item  $\|\omega(t)\|_{L^{p,q}}+\|\partial_{z}\omega\|_{L_{t}^2(L^{p,q})}
  \leq
  C\bigl(\|\omega_{0}\|_{L^{p,q}}+\sqrt{t}\,\| \frac{b_0^{2}}{r}\|_{L^{p,q}}\bigr)e^{CA_0(t)}$  ,
   \item $ \|b(t)\|_{L^{p,q}}+\|\partial_{z}b\|_{L^{2}_t(L^{p,q})}
   \leq
   C\|b_0 \|_{L^{p,q}}e^{CA_{0}(t)}$,
     \item $ \|\frac{b^2}{r}(t)\|_{L^{p,q}}+\|\partial_{z}\frac{b^2}{r}\|_{L^{2}_t(L^{p,q})}
   \leq
   C\|\frac{b_0^2}{r}\|_{L^{p,q}}e^{CA_{0}(t)}$.
   \end{itemize}
  \item if $2<  p<\infty$, $1 \leq q \leq p$, then
  \begin{itemize}
  \item  $\|\omega(t)\|_{L^{p,q}}
  \leq
  C\bigl(\|\omega_{0}\|_{L^{p,q}}+\sqrt{t}\,\| \frac{b_0^{2}}{r}\|_{L^{p,q}}\bigr)e^{CA_0(t)}$  ,
   \item $ \|b(t)\|_{L^{p,q}}
   \leq
   C\|b_0 \|_{L^{p,q}}e^{CA_{0}(t)}$,
     \item $ \|\frac{b^2}{r}(t)\|_{L^{p,q}}
   \leq
   C\|\frac{b_0^2}{r}\|_{L^{p,q}}e^{CA_{0}(t)}$.
   \end{itemize}
\end{itemize}
In particular, if $\omega_{0},\,\frac{\omega_0}{r}\in L^{\frac{3}{2},1}$, $b_{0}\in L^{3, 2}$, and $r^{-1} b_{0}\in L^{\frac{3}{2}, 1}\cap L^{3, 2}$, we have
\begin{equation}\label{est-omega-32-1}
\begin{split}
&\|\omega\|_{L^{\infty}_t(L^{\frac{3}{2},1})}+\|\partial_{z}\omega\|_{L_{t}^{2}(L^{\frac{3}{2},1})}\leq C\bigl(\|\omega_{0}\|_{L^{\frac{3}{2},1}}+\sqrt{t}\,\| \frac{b_0^{2}}{r}\|_{L^{\frac{3}{2},1}}\bigr)e^{CA_0(t)},\\
&\|r^{-1}\omega\|_{L^{\infty}_t(L^{\frac{3}{2},1})}
+\|r^{-1}\partial_z\omega\|_{L^2_t(L^{{3\over 2},1})}
\leq
C\|r^{-1}\omega_0\|_{L^{{3\over 2},1}}+\sqrt{t}\|r^{-1}b_{0}\|_{L^{3,2}}^2,\\
& \|\frac{b}{r}\|_{L^{\infty}_t(L^{\frac{3}{2},1})}+\|\partial_{z}\frac{b}{r}\|_{L^{2}_t(L^{\frac{3}{2},1})}
   \leq C\|\frac{b_0}{r}\|_{L^{\frac{3}{2},1}},\\
   & \|\frac{b^2}{r}\|_{L^{\infty}_t(L^{\frac{3}{2},1})}+\|\partial_{z}\frac{b^2}{r}\|_{L^{2}_t(L^{\frac{3}{2},1})}
   \leq C\|\frac{b_0^2}{r}\|_{L^{\frac{3}{2},1}}e^{CA_{0}(t)},\\
  &\|\frac{b}{r}\|_{L^{\infty}_t(L^{3, 2})}   \leq    C\|\frac{b_0}{r}\|_{L^{3, 2}},\quad \|b\|_{L^{\infty}_t(L^{3, 2})}
   \leq    C\|b_0 \|_{L^{3, 2}}e^{CA_{0}(t)}.
\end{split}
\end{equation}
\end{col}


Below we give more higher-order estimates which will be used in the proof of uniqueness.

\begin{prop}\label{prop3.3}
Assume that the initial data $(\omega_0,b_{0})$ satisfies
\begin{align*}
&\omega_0\in L^{3,1}\cap L^{{3\over2},1},\quad
r^{-1}\omega_0 \in L^{{3\over2},1}, \quad  b_{0}\in L^2\cap L^{3, 2},\quad r^{-1}b_{0}\in L^{3, 2}\cap \dot{H}^1.
\end{align*}
 Let $(\omega,b)$ be a regular solution of the system \eqref{1.4}.
Then
\begin{equation}\label{est-nabla-br-1}
\|\nabla\frac{b}{r}(t)\|_{L^{2}}^2+\|\partial_{z}\nabla \frac{b}{r}\|_{L^{2}_{t}(L^{2})}^2
\leq C\|\nabla\frac{b_{0}}{r}\|_{L^{2}}^2 
\exp\{CA_0(t)+C A_1(t) e^{CA_{0}(t))}\},
\end{equation}
where
\begin{align*}
&A_1(t)\triangleq t\|\omega_{0}\|_{L^{3,1}}^2+t^2\|\nabla\frac{b_0}{r}\|_{L^{2}}^2\|b_0\|_{L^{2}}^2+\|\omega_{0}\|_{L^{\frac{3}{2},1}}^2+t\,\|b_0\|_{L^{3,2}}^2 \|\frac{b_0}{r}\|_{L^{3,2}}^2.
\end{align*}
\end{prop}
\begin{proof}
Multiplying both sides of the second of \eqref{1.5} by $-\Delta\frac{b}{r}$, then integrating the resulting equation, we get
\begin{equation}\label{3.21}
  \begin{split}
&\|\nabla\frac{b}{r}\|_{L^{2}}^{2}(t)+\|\partial_{z}\nabla \frac{b}{r}\|_{L^{2}_{t}(L^{2})}^{2}
=2\pi\int_{\mathbb{R}^2_{+}}(u^r\partial_{r}\frac{b}{r}+u^z\partial_{z}\frac{b}{r})(\frac{1}{r}\partial_{r}(r\partial_{r}\frac{b}{r})
+\partial_{z}^2\frac{b}{r})rdrdz\\
&=2\pi\int_{\mathbb{R}^2_{+}}u^r\partial_{r}\frac{b}{r}\partial_{r}(r\partial_{r}\frac{b}{r})drdz
+2\pi\int_{\mathbb{R}^2_{+}}u^z\partial_{z}\frac{b}{r}\partial_{r}(r\partial_{r}\frac{b}{r})drdz\\
&\qquad+2\pi\int_{\mathbb{R}^2_{+}}(u^r\partial_{r}\frac{b}{r}+u^z\partial_{z}\frac{b}{r})\partial_{z}^2\frac{b}{r}rdrdz\eqdefa I_{1}+I_{2}+I_{3}.
  \end{split}
\end{equation}
Note that $\partial_{r}u^r=-\frac{u^r}{r}-\partial_{z}u^z$, so we find
\begin{equation*}
  \begin{split}
 I_{1}&=\pi\int_{\mathbb{R}^2_{+}}\frac{u^r}{r}\partial_{r}(r\partial_{r}\frac{b}{r})^{2}drdz
 =\pi\int_{\mathbb{R}^2_{+}}\frac{u^r}{r^2}r^2(\partial_{r}\frac{b}{r})^{2}drdz
 -\pi\int_{\mathbb{R}^2_{+}}\frac{1}{r}\partial_{r}u^rr^2(\partial_{r}\frac{b}{r})^{2}drdz\\
&=2\pi\int_{\mathbb{R}^2_{+}}\frac{u^r}{r}(\partial_{r}\frac{b}{r})^{2}rdrdz+\pi\int_{\mathbb{R}^2_{+}} \partial_{z}u^z (\partial_{r}\frac{b}{r})^{2}rdrdz\\
&=2\pi\int_{\mathbb{R}^2_{+}}\frac{u^r}{r}(\partial_{r}\frac{b}{r})^{2}rdrdz
-2\pi\int_{\mathbb{R}^2_{+}}u^z(\partial_{r}\frac{b}{r})\partial_{z}\partial_{r}\frac{b}{r}rdrdz,
  \end{split}
\end{equation*}
which implies
\begin{align*}
|I_{1}|
\lesssim\|\frac{u^r}{r}\|_{L^\infty}\|\partial_{r}\frac{b}{r}\|_{L^2}^2
+\|u\|_{L^{\infty}}\|\partial_z\partial_r\frac{b}{r}\|_{L^2}\|\partial_{r}\frac{b}{r}\|_{L^2}.
\end{align*}
By virtue of $\partial_{r}u^z=\partial_{z}u^r-\omega$ and integration by parts, the integral $I_{2}$ satisfies
\begin{align*}
I_{2}&=-2\pi\int_{\mathbb{R}^2_{+}}u^z\partial_{z}\partial_{r}\frac{b}{r}\partial_{r}\frac{b}{r}rdrdz
-2\pi\int_{\mathbb{R}^2_{+}}\partial_{r}u^z\partial_{z}\frac{b}{r}\partial_{r}\frac{b}{r}rdrdz\\
&=-\int_{\mathbb{R}^3}u^z\partial_{z}\partial_{r}\frac{b}{r}\partial_{r}\frac{b}{r}dx+\int_{\mathbb{R}^3}(\omega-\partial_{z}u^r)\partial_{z}\frac{b}{r}\partial_{r}\frac{b}{r}\mathrm{d}x,
\end{align*}
which along with the facts $\|\partial_zu^r\|_{L^{3}} \lesssim \|\partial_z\omega\|_{L^{\frac{3}{2}}}$ and $\|u\|_{L^\infty}
\lesssim \|\omega\|_{L^{3,1}}$ gives rise to
\begin{equation*}
  \begin{split}
|I_{2}|
&\lesssim\|u^z\|_{L^{\infty}}\|\partial_z\partial_r\frac{b}{r}\|_{L^2}\|\partial_{r}\frac{b}{r}\|_{L^2}
+(\|\omega\|_{L^{3}}+\|\partial_zu^r\|_{L^{3}})\|\partial_{z}\frac{b}{r}\|_{L^{6}}\|\partial_{r}\frac{b}{r}\|_{L^{2}}\\
&\lesssim (\|\omega\|_{L^{3,1}}+\|\partial_z\omega\|_{L^{\frac{3}{2}}})\|\partial_{z}\nabla\frac{b}{r}\|_{L^{2}}\|\nabla\frac{b}{r}\|_{L^{2}}.
 \end{split}
\end{equation*}
By H\"{o}lder's inequality, we have
\begin{align*}
|I_{3}|&\lesssim\|u\|_{L^{\infty}}\|\nabla\frac{b}{r}\|_{L^{2}}\|\partial_{z}\nabla\frac{b}{r}\|_{L^{2}} \lesssim \|\omega\|_{L^{3,1}} \|\partial_{z}\nabla\frac{b}{r}\|_{L^{2}}\|\nabla\frac{b}{r}\|_{L^{2}}.
\end{align*}
Substituting $I_{1}-I_{3}$ estimates into \eqref{3.21}, from the fact $\|u\|_{L^\infty}
\lesssim \|\omega\|_{L^{3,1}}$, we obtain
\begin{equation*}
\begin{split}
&\|\nabla\frac{b}{r}(t)\|_{L^{2}}^2+\|\partial_{z}\nabla \frac{b}{r}\|_{L^{2}_{t}(L^{2})}^2\\
&\lesssim \|\frac{u^r}{r}\|_{L^\infty}\|\nabla\frac{b}{r}\|_{L^2}^2+\|\omega\|_{L^{3,1}}\|\nabla\frac{b}{r}\|_{L^{2}}\|\partial_{z}\nabla\frac{b}{r}\|_{L^{2}}+ \|\partial_z\omega\|_{L^{\frac{3}{2}}}\|\partial_{z}\nabla\frac{b}{r}\|_{L^{2}}\|\nabla\frac{b}{r}\|_{L^{2}},
 \end{split}
\end{equation*}
which along with Young's inequality implies
\begin{equation*}
\begin{split}
&\|\nabla\frac{b}{r}(t)\|_{L^{2}}^2+\|\partial_{z}\nabla \frac{b}{r}\|_{L^{2}_{t}(L^{2})}^2\leq C(\|\frac{u^r}{r}\|_{L^\infty}+\|\omega\|_{L^{3,1}}^2+\|\partial_z\omega\|_{L^{\frac{3}{2}}}^2) \|\nabla\frac{b}{r}\|_{L^2}^2.
 \end{split}
\end{equation*}
Hence, applying Gronwall's inequality gives rise to
\begin{equation}\label{3.22}
  \begin{split}
&\|\nabla\frac{b}{r}(t)\|_{L^{2}}^2+\|\partial_{z}\nabla \frac{b}{r}\|_{L^{2}_{t}(L^{2})}^2\leq C\|\nabla\frac{b_{0}}{r}\|_{L^{2}}^2 
\exp\{C\int_{0}^{t}(\|\frac{u^r}{r}\|_{L^\infty}+\|\omega\|_{L^{3,1}}^2+\|\partial_z\omega\|_{L^{\frac{3}{2}}}^2)d\tau\}.
 \end{split}
\end{equation} 
Thanks to Corollary \ref{cro3.1}, Proposition \ref{prop3.2}, and the Sobolev embedding $\dot{H}^1(\mathbb{R}^3)\hookrightarrow L^{6, 2}(\mathbb{R}^3)$ (see \cite{Peetre-1966, Bergh-1976, Tartar-1998}), we know that
\begin{align*}
\|\frac{u^r}{r}\|_{L^{\infty}_t(L^{\infty})}
&\leq
C(\|\Omega_{0}\|_{L^{\frac{3}{2},1}}+\sqrt{t}\|\Gamma_{0}\|_{L^{3,2}}^2),\\
\|\omega\|_{L^{\infty}_t(L^{3,1})}
&\leq
C\bigl(\|\omega_{0}\|_{L^{3,1}}+\sqrt{t}\|\frac{b_0}{r}\|_{L^{6,2}}\|b_0\|_{L^{6,2}}\bigr)e^{CA_{0}(t)}\\
&\leq
C\bigl(\|\omega_{0}\|_{L^{3,1}}+\sqrt{t}\|\nabla\frac{b_0}{r}\|_{L^{2}}\|\nabla\,b_0\|_{L^{2}}\bigr)e^{CA_{0}(t)},\\
\|\partial_{z}\omega\|_{L_{t}^{2}(L^{\frac{3}{2},1})}^2&\leq C\bigl(\|\omega_{0}\|_{L^{\frac{3}{2},1}}^2+t\,\|b_0\|_{L^{3,2}}^2 \|\frac{b_0}{r}\|_{L^{3,2}}^2\bigr)e^{CA_0(t)}.
\end{align*}
Substituting the above inequalities into \eqref{3.22}, we get \eqref{est-nabla-br-1}, which concludes the proof of Proposition \ref{prop3.3}.
\end{proof}
\begin{rmk}\label{rmk-embedding-1}
Thanks to Theorem 5.3.1. in \cite{Bergh-1976}, we have the following interpolation inequality
$$\|b_0\|_{L^{3, 2}(\mathbb{R}^3)} \lesssim \|b_0\|_{L^{\frac{3}{2}, 2}(\mathbb{R}^3)}^{\frac{1}{3}}\|b_0\|_{L^{6, 2}(\mathbb{R}^3)}^{\frac{2}{3}} \lesssim  \|b_0\|_{L^{\frac{3}{2}}(\mathbb{R}^3)}^{\frac{1}{3}}\|\nabla\,b_0\|_{L^{2}(\mathbb{R}^3)}^{\frac{2}{3}}, $$
where we used the embedding $\dot{H}^1(\mathbb{R}^3)\hookrightarrow L^{6, 2}(\mathbb{R}^3)$ (see \cite{Peetre-1966, Tartar-1998}), which implies that
$$  
L^{\frac{3}{2}}(\mathbb{R}^3) \cap \dot{H}^1 (\mathbb{R}^3)\subset
L^{3, 2}(\mathbb{R}^3) .
$$
\end{rmk}
\begin{prop}\label{prop3.4}
Assume that the initial data $(\omega_0,b_{0})$ satisfies
\begin{align*}
&\omega_0\in L^{3, 1},\quad
r^{-1}\omega_0\in L^{{3\over2},1}, \quad  b_{0}\in L^{3, 2} \cap \dot{H}^1,\quad {r}^{-1}b_{0}\in L^{3, 2} \cap \dot{H}^1.
\end{align*}
Assume that $(\omega,b)$ is a regular solution of the system \eqref{1.4}, then there holds
\begin{equation}\label{3.49}
\begin{split}
&\|\nabla\,b\|_{L^{2}}^{2}+\|\partial_{z}\nabla\,b\|_{L^2_t(L^{2})}^{2}
\leq C\|\nabla\,b_0\|_{L^2}^2\exp\{CA_0(t)+CA_2(t)e^{CA_0(t)}\},
\end{split}
\end{equation}
where
\begin{equation*}
\begin{split}
A_2(t)\eqdefa &t\|\omega_{0}\|_{L^{3,1}}+{t}^{\frac{3}{2}}\,\|\nabla( {r}^{-1}{b_0})\|_{L^{2}}\|\nabla\,b_0\|_{L^{2}}+t\|\omega_{0}\|_{L^{3,1}}^2\\
&\quad+t^2\,\|\nabla({r}^{-1}{b_0})\|_{L^{2}}^2\|\nabla\,b_0\|_{L^{2}}^2+\|\omega_{0}\|_{L^{\frac{3}{2},1}}^2
+t\,\|{r}^{-1} b_0\|_{L^{3,2}}^2\|b_0\|_{L^{3,2}}^2.
\end{split}
\end{equation*}
\end{prop}
\begin{proof}
Acting the operator $\partial_r$ to the second equation of \eqref{1.4} yields
\begin{equation}\label{14-1}
\begin{split}
\partial_{t}\partial_{r}b+(u\cdot\nabla)\partial_{r}b-\partial_{z}^2\partial_{r}b
=\partial_{r}(\frac{u^rb}{r})-\partial_{r}u^r\partial_{r}b-\partial_{r}u^z\partial_{z}b.
\end{split}
\end{equation}
Multiplying \eqref{14-1} by $\partial_r\,b$, then integrating the result equation gives rise to
\begin{equation}\label{3.34}
\begin{split}
\frac{1}{2}\frac{\mathrm{d}}{\mathrm{d}t}\|\partial_{r}b\|_{L^{2}}^{2}+\|\partial_{z}\partial_{r}b\|_{L^{2}}^{2}
&=\int_{\mathbb{R}^3}\partial_{r}(\frac{u^rb}{r})\partial_{r}b\,dx-\int_{\mathbb{R}^3}\partial_{r}u^r(\partial_{r}b)^2\,dx
-\int_{\mathbb{R}^3}\partial_{r}u^z\partial_{z}b\partial_{r}b\,dx\\
&\eqdefa K_{1}+K_2+K_3.
\end{split}
\end{equation}
By using the identity $\partial_{r}u^r=-\frac{u^r}{r}-\partial_{z}u^z$, we have
\begin{equation*}
\begin{split}
K_{1}&=\int_{\mathbb{R}^3}-\frac{u^r}{r}\frac{b}{r}\partial_{r}b\,dx
+\int_{\mathbb{R}^3}\partial_{r}u^r\frac{b}{r}\partial_{r}b\,dx
+\int_{\mathbb{R}^3}\frac{u^r}{r}(\partial_{r}b)^2\,dx\\
&=-2\int_{\mathbb{R}^3}\frac{u^r}{r}\frac{b}{r}\partial_{r}b\,dx
-\int_{\mathbb{R}^3}\partial_{z}u^z\frac{b}{r}\partial_{r}b\,dx+\int_{\mathbb{R}^3}\frac{u^r}{r}(\partial_{r}b)^2\,dx\\
&=-2\int_{\mathbb{R}^3}\frac{u^r}{r}\frac{b}{r}\partial_{r}b\,dx
+\int_{\mathbb{R}^3}u^z\partial_{z}\frac{b}{r}\partial_{r}b\,dx
+\int_{\mathbb{R}^3}u^z\frac{b}{r}\partial_{z}\partial_{r}b\,dx+\int_{\mathbb{R}^3}\frac{u^r}{r}(\partial_{r}b)^2\,dx,
\end{split}
\end{equation*}
which along with H\"{o}lder inequality leads to
\begin{equation*}
\begin{split}
|K_{1}|
&\leq 2\|\frac{u^r}{r}\|_{L^{\infty}}\|\frac{b}{r}\|_{L^2}\|\partial_{r}b\|_{L^2}+\|\frac{u^r}{r}\|_{L^{\infty}}\|\partial_{r}b\|_{L^2}^2\notag\\
&\qquad+\|u^z\|_{L^\infty}\|\partial_{z}\frac{b}{r}\|_{L^2}\|\partial_{r}b\|_{L^2}+\|\partial_{z}\partial_{r}b\|_{L^2}\|u^z\|_{L^\infty}\|\frac{b}{r}\|_{L^2}.
\end{split}
\end{equation*}
Along the same line, the bound of $K_2$ yields
\begin{equation*}
\begin{split}
K_2&=\int_{\mathbb{R}^3}\frac{u^r}{r}(\partial_{r}b)^2\,dx
+\int_{\mathbb{R}^3}\partial_{z}u^z(\partial_{r}b)^2\,dx
=\int_{\mathbb{R}^3}\frac{u^r}{r}(\partial_{r}b)^2\,dx+2\int_{\mathbb{R}^3}u^z\partial_{z}\partial_{r}b\partial_{r}b\,dx\\
&\leq\|\frac{u^r}{r}\|_{L^\infty}\|\partial_{r}b\|_{L^2}^2+\|u^z\|_{L^\infty}\|\partial_{z}\partial_{r}b\|_{L^2}\|\partial_{r}b\|_{L^2}.
\end{split}
\end{equation*}
By the definition $\partial_{r}u^z=\omega-\partial_{z}u^r$, we get $K_3=\int_{\mathbb{R}^3}\omega \partial_{z}b\partial_{r}b\,dx-\int_{\mathbb{R}^3}\partial_{z}u^r\partial_{z}b\partial_{r}b\,dx$, which follows
\begin{equation*}
\begin{split}
|K_3|&\lesssim (\|\omega\|_{L^3} +\|\partial_{z}u^r\|_{L^3})\|\partial_{z}b\|_{L^6}\|\partial_{r}b\|_{L^2}.
\end{split}
\end{equation*}
Due to the fact $\|\partial_zu^r\|_{L^3} \leq C\|\partial_z\omega\|_{L^{\frac{3}{2}, 3}}$, we infer
\begin{equation*}
\begin{split}
|K_3|&\lesssim (\|\omega\|_{L^3}  +\|\partial_z\omega\|_{L^{\frac{3}{2}, 3}})\|\partial_{z}\nabla\,b\|_{L^2} \|\partial_{r}b\|_{L^2}.
\end{split}
\end{equation*}
Inserting the estimates of $K_1$-$K_3$ into \eqref{3.34}, we obtain
\begin{equation}\label{3.35}
\begin{split}
&\frac{1}{2}\frac{\mathrm{d}}{\mathrm{d}t}\|\partial_{r}b\|_{L^{2}}^{2}+\|\partial_{z}\partial_{r}b\|_{L^{2}}^{2}
\leq C(\|{r}^{-1}{u^r}\|_{L^{\infty}}+\|u^z\|_{L^\infty}^2)\|(\partial_{r}b,\,{r}^{-1}{b})\|_{L^2}^2\\
&\qquad+C(\|u^z\|_{L^\infty}+\|\omega\|_{L^3}  +\|\partial_z\omega\|_{L^{\frac{3}{2}, 3}})\|(\partial_{r}b,\,{r}^{-1}{b})\|_{L^2}\|\partial_z(\partial_{r}b,\,\partial_z b,\,{r}^{-1}{b})\|_{L^2},
\end{split}
\end{equation}
which along with \eqref{3.1} gives rise to
\begin{equation}\label{3.36}
\begin{split}
&\frac{1}{2}\frac{\mathrm{d}}{\mathrm{d}t}\|(\partial_{r}b,\,{r}^{-1}{b})\|_{L^{2}}^{2}+\|\partial_{z}(\partial_{r}b,\,{r}^{-1}{b})\|_{L^{2}}^{2}
\leq C(\|{r}^{-1}{u^r}\|_{L^{\infty}}+\|u^z\|_{L^\infty}^2)\|(\partial_{r}b,\,{r}^{-1}{b})\|_{L^2}^2\\
&\qquad+C(\|u^z\|_{L^\infty}+\|\omega\|_{L^3}  +\|\partial_z\omega\|_{L^{\frac{3}{2}, 3}})\|(\partial_{r}b,\,{r}^{-1}{b})\|_{L^2}\|\partial_z(\partial_{r}b,\,\partial_z b,\,{r}^{-1}{b})\|_{L^2}.
\end{split}
\end{equation}

We may repeat the above argument to get the estimate of $\|\partial_{z}b \|_{L^{2}}$. In fact, acting the operator $\partial_z$ to the second of the system \eqref{1.4} yields
\begin{equation}\label{15-1}
\begin{split}
\partial_{t}\partial_{z}b+(u\cdot\nabla)\partial_{z}b-\partial_{z}^2\partial_{z}b
=\partial_{z}u^r({r}^{-1}b-\partial_{r}b)+{r}^{-1} u^r\partial_{z}b -\partial_{z}u^z\partial_{z}b.
\end{split}
\end{equation}
Multiply \eqref{15-1} by $\partial_z\,b$, then integrating by part gives
\begin{equation*}
\begin{split}
&\frac{1}{2}\frac{\mathrm{d}}{\mathrm{d}t}\|\partial_{z}b\|_{L^{2}}^{2}+\|\partial_{z}^2b\|_{L^{2}}^{2}\\
&=\int_{\mathbb{R}^3}\partial_{z}u^r({r}^{-1}b-\partial_{r}b)\partial_{z}b\,\mathrm{d}x+\int_{\mathbb{R}^3}{r}^{-1} u^r (\partial_{z}b)^2\,\mathrm{d}x
-\int_{\mathbb{R}^3}\partial_{z}u^z(\partial_{z}b)^2\,\mathrm{d}x\\
&=-\int_{\mathbb{R}^3}u^r\partial_{z}[({r}^{-1}b-\partial_{r}b)\partial_{z}b]\,\mathrm{d}x+\int_{\mathbb{R}^3}{r}^{-1} u^r (\partial_{z}b)^2\,\mathrm{d}x
+2\int_{\mathbb{R}^3}u^z\partial_{z}b\partial_{z}^2b \,\mathrm{d}x.
\end{split}
\end{equation*}
Hence, we have
\begin{equation*}
\begin{split}
\frac{1}{2}\frac{d}{dt}\|\partial_{z}b\|_{L^{2}}^{2}+\|\partial_{z}^2b\|_{L^{2}}^{2}
&\leq\|u^r\|_{L^\infty}\Big(\|\partial_{z}({r}^{-1}b-\partial_{r}b)\|_{L^2}\|\partial_{z}b\|_{L^2}
+\|({r}^{-1}b-\partial_{r}b)\|_{L^2}\|\partial_{z}^2b\|_{L^2}\Big)\\
&\quad+\Big(\|{r}^{-1} u^r\|_{L^\infty}+2 \|u^z\|_{L^\infty}^2\Big)\|\partial_{z}b\|_{L^2}^2,
\end{split}
\end{equation*}
from which, together with \eqref{3.36}, one has
\begin{equation*}
\begin{split}
&\frac{1}{2}\frac{\mathrm{d}}{\mathrm{d}t}\|(\partial_{r}b,\,{r}^{-1}{b}, \,\partial_zb)\|_{L^{2}}^{2}+\|\partial_{z}(\partial_{r}b,\,{r}^{-1}{b}, \,\partial_zb)\|_{L^{2}}^{2}\\
&\leq C\Big(\|{r}^{-1}{u^r}\|_{L^{\infty}}+\|u^z\|_{L^{\infty}}^2\Big)\|(\partial_{r}b,\,{r}^{-1}{b}, \,\partial_zb)\|_{L^2}^2\\
&\quad+C\Big(\|u\|_{L^\infty}+\|\omega\|_{L^3}  +\|\partial_z\omega\|_{L^{\frac{3}{2}, 3}}\Big)\|(\partial_{r}b,\,{r}^{-1}{b}, \,\partial_zb)\|_{L^2}\|\partial_z(\partial_{r}b,\,{r}^{-1}{b}, \,\partial_zb)\|_{L^2}.
\end{split}
\end{equation*}
Thanks to Young's inequality, we find
\begin{equation*}
\begin{split}
&\frac{\mathrm{d}}{\mathrm{d}t}\|(\partial_{r}b,\,{r}^{-1}{b}, \,\partial_zb)\|_{L^{2}}^{2}+\|\partial_{z}(\partial_{r}b,\,{r}^{-1}{b}, \,\partial_zb)\|_{L^{2}}^{2}\\
&\leq C\Big(\|({r}^{-1}{u^r},\,u^z)\|_{L^{\infty}}+\|u\|_{L^\infty}^2+\|\omega\|_{L^3}^2 +\|\partial_z\omega\|_{L^{\frac{3}{2}, 3}}^2\Big)\|(\partial_{r}b,\,{r}^{-1}{b}, \,\partial_zb)\|_{L^2}^2,
\end{split}
\end{equation*}
which follows that
\begin{equation}\label{3.47}
\begin{split}
&\|(\partial_{r}b,\,{r}^{-1}{b}, \,\partial_zb)\|_{L^{2}}^{2}+\|\partial_{z}(\partial_{r}b,\,{r}^{-1}{b}, \,\partial_zb)\|_{L^2_t(L^{2})}^{2}\\
&\leq C\|(\partial_{r}b_0,\,{r}^{-1}{b}_0, \,\partial_zb_0)\|_{L^2}^2\\
&\qquad \times\exp\Big\{C\int_0^t(\|({r}^{-1}{u^r},\,u^z)\|_{L^{\infty}}+\|u\|_{L^\infty}^2+\|\omega\|_{L^3}^2) \,\mathrm{d}\tau +C\|\partial_z\omega\|_{L^2_t(L^{\frac{3}{2}, 3})}^2\Big\}.
\end{split}
\end{equation}
Thanks to \eqref{3.17}, we know that
\begin{equation}\label{3.47-111}
\begin{split}
&\|u\|_{L^\infty}+\|\omega\|_{L^{3}}\leq C\|\omega\|_{L^{3,1}}\leq   C\bigl(\|\omega_{0}\|_{L^{3,1}}+\sqrt{t}\,\| {r}^{-1}{b_0^{2}}\|_{L^{3,1}}\bigr)e^{CA_0(t)},\\
&\|\omega(t)\|_{L^{\frac{3}{2},1}}+\|\partial_{z}\omega\|_{L_{t}^2(L^{\frac{3}{2},1})}
  \leq C\bigl(\|\omega_{0}\|_{L^{\frac{3}{2},1}}+\sqrt{t}\,\|{r}^{-1}{b_0^{2}}\|_{L^{\frac{3}{2},1}}\bigr)e^{CA_0(t)}.
  \end{split}
\end{equation}
Hence, inserting \eqref{3.47-111} into \eqref{3.47} yields
\begin{equation*}
\begin{split}
&\|\nabla\,b\|_{L^{2}}^{2}+\|\partial_{z}\nabla\,b\|_{L^2_t(L^{2})}^{2}
\leq C\|\nabla\,b_0\|_{L^2}^2e^{C A_0(t)}\\
&\quad\times\exp\{C\bigl(t\|\omega_{0}\|_{L^{3,1}}+t^{\frac{3}{2}}\,\| {r}^{-1}{b_0^{2}}\|_{L^{3,1}}\\
&\qquad\qquad\quad+t\|\omega_{0}\|_{L^{3,1}}^2+t^2\,\| {r}^{-1}{b_0^{2}}\|_{L^{3,1}}^2+\|\omega_{0}\|_{L^{\frac{3}{2},1}}^2+t\,\|{r}^{-1}{b_0^{2}}\|_{L^{\frac{3}{2},1}}^2\bigr)e^{CA_0(t)}\}
\end{split}
\end{equation*}
which implies \eqref{3.49}.

Hence, we complete the proof of Proposition \ref{prop3.4}.
\end{proof}

\begin{prop}\label{prop3.5}
Let the initial data $(\omega_0, b_{0})$ satisfy
\begin{align*}
&\omega_0\in L^{3, 1},\,
r^{-1}\omega_0\in L^{{3\over2},1}, \,\partial_r\omega_0 \in L^{\frac{3}{2}},\,  b_{0}\in L^{3, 2}\cap \dot{H}^1,\, r^{-1}b_{0}\in {H}^1.
\end{align*}
Let
$(\omega,b)$ a regular solution
of the system \eqref{1.4}.
Then
\begin{equation}\label{est-romega-1}
\|\partial_r\omega(t)\|_{L^\infty_t(L^{{\frac{3}{2}}})}+
\|\partial_z\partial_r\omega\|_{L^2_t(L^{{\frac{3}{2}}})}
\leq
C(t,\omega_0,b_{0}).
\end{equation}
\end{prop}
\begin{proof}
Acting the operator $\partial_r$ to the first equation of \eqref{1.4} yields
\begin{equation}\label{omega-eqns-r-1}
\begin{split}
\partial_{t} \partial_r\omega+(u\cdot\nabla)\partial_r\omega-\partial_{z}^{2}\partial_r\omega=-\partial_{z}\partial_r(\frac{b^{2}}{r})+\partial_r(\frac{ u^{r}}{r}\omega)+(\frac{u^r}{r}+\partial_zu^z)\partial_r\omega+g_1,
\end{split}
\end{equation}
where $g_{1}:=-\partial_zu^r\partial_z\omega+\omega\partial_z\omega$. Multiplying the equation above by  $\partial_r\omega$ by $|\partial_r\omega|^{\frac 12}{\rm sign} \,\,(\partial_r\omega)$ and integrating the result equation, we obtain
\begin{equation*} 
\begin{split}
&{2\over 3}{\mathrm{d}\over \mathrm{d}t}\|\partial_r\omega\|_{L^{\frac{3}{2}}}^{\frac{3}{2}}
+{8\over 9}
\|\partial_z|\partial_r\omega|^{3\over4}\|_{L^2}^2
\leq
2\|{u^r\over r}\|_{L^\infty}\|\partial_r\omega\|_{L^{\frac{3}{2}}}^{\frac{3}{2}}
+\int_{\mathbb{R}^3}\partial_zu^z|\partial_r\omega|^{\frac{3}{2}}\mathrm{d}x
\\
&\quad
+\bigl(2\|{u^r\over r}\|_{L^\infty}\|{\omega\over r}\|_{L^{\frac{3}{2}}}
+\|\partial_zu^z{\omega\over r}\|_{L^{\frac{3}{2}}}+\|g_{1}\|_{L^{\frac{3}{2}}}\bigr)
\|\partial_r\omega\|_{L^{\frac{3}{2}}}^{\frac{1}{2}}+\int_{\mathbb{R}^3}\partial_{z}\partial_{r}(\frac{b^2}{r})|\partial_r\omega|^{\frac{1}{2}}\mathrm{d}x,
\end{split}
\end{equation*}
Integrating by parts and using the Cauchy-Schwartz inequality, we get
\begin{equation*} 
\begin{split}
\int_{\mathbb{R}^3}\partial_zu^z|\partial_r\omega|^{\frac{3}{2}}\mathrm{d}x=-2\int_{\mathbb{R}^3} u^z
|\partial_r\omega|^{\frac{3}{4}}
\partial_z|\partial_r\omega|^{\frac{3}{4}}\mathrm{d}x
\le
2\|u^z\|_{L^\infty}\Big\|\partial_z|
\partial_r\omega|^{\frac{3}{4}}\Big\|_{L^2}
\|\partial_r\omega\|_{L^{\frac{3}{2}}}^{\frac{3}{4}}
\end{split}
\end{equation*}
and
\begin{equation*} 
\begin{split}
&\int_{\mathbb{R}^3}\partial_{z}\partial_{r}(\frac{b^2}{r})|\partial_r\omega|^{\frac{1}{2}}\,dx
=-\int_{\mathbb{R}^3}\partial_{z}(\frac{b^2}{r^2})|\partial_r\omega|^{\frac{1}{2}}\,dx+2\int_{\mathbb{R}^3}\frac{b}{r}\partial_{z}\partial_{r}b|\partial_r\omega|^{\frac{1}{2}}\mathrm{d}x+2\int_{\mathbb{R}^3}\partial_{z}\frac{b}{r}\partial_{r}b|\partial_r\omega|^{\frac{1}{2}}\mathrm{d}x\notag\\
&\lesssim
\Bigl(\|\partial_{z}(\frac{b^2}{r^2})\|_{L^{\frac{3}{2}}}+\|\frac{b}{r}\|_{L^{6}}\|\partial_{z}\partial_{r}b\|_{L^2}+\|\partial_{z}\frac{b}{r}\|_{L^{6}}\|\partial_{r}b\|_{L^2}\Bigr)\|\partial_r\omega\|_{L^{\frac{3}{2}}}^{\frac{1}{2}}\\
&\lesssim
\Bigl(\|\partial_{z}(\frac{b^2}{r^2})\|_{L^{\frac{3}{2}}}+\|\nabla\frac{b}{r}\|_{L^{2}}\|\partial_{z}\partial_{r}b\|_{L^2}+\|\partial_{z}\nabla\frac{b}{r}\|_{L^{2}}\|\partial_{r}b\|_{L^2}\Bigr)\|\partial_r\omega\|_{L^{\frac{3}{2}}}^{\frac{1}{2}}.
\end{split}
\end{equation*}
As a consequence, we have
\begin{equation}\label{3.37}
\begin{split}
&\frac{d}{dt}\|\partial_r\omega\|_{L^{\frac{3}{2}}}^{\frac{3}{2}}+
\|\partial_z|\partial_r\omega|^{\frac{3}{4}}\|_{L^2}^2\\
&\lesssim
\Bigl(\|{u^r\over r}\|_{L^\infty}+\|u^z\|_{L^\infty}^2\Bigr)
\|\partial_r\omega\|_{L^{\frac{3}{2}}}^{\frac{3}{2}}
+\Bigl(\|\partial_zu^z{\omega\over r}\|_{L^{\frac{3}{2}}}
+\|{u^r\over r}\|_{L^\infty}\|{\omega\over r}\|_{L^{\frac{3}{2}}}
+\|g_{1}\|_{L^{\frac{3}{2}}}\Bigr)
\|\partial_r\omega\|_{L^{\frac{3}{2}}}^{\frac{1}{2}}\\
&\qquad +\Bigl(\|\partial_{z}(\frac{b^2}{r^2})\|_{L^{\frac{3}{2}}}+\|\nabla\frac{b}{r}\|_{L^{2}}\|\partial_{z}\partial_{r}b\|_{L^2}+\|\partial_{z}\nabla\frac{b}{r}\|_{L^{2}}\|\partial_{r}b\|_{L^2}\Bigr)\|\partial_r\omega\|_{L^{\frac{3}{2}}}^{\frac{1}{2}}.
\end{split}
\end{equation}
Thanks to H\"{o}lder's inequality, Proposition \ref{prop2.2}, and the interpolation inequality, we have
\begin{equation*}
\begin{split}
\|\partial_zu^z{\omega\over r}\|_{L^{\frac{3}{2}}}
&\lesssim
\|{\omega\over r}\|_{L^{\frac{3}{2}}_h(L^\infty_v)}
\|\partial_zu^z\|_{L^\infty_h(L^{\frac{3}{2}}_v)}\lesssim
\|{\omega\over r}\|_{L^{\frac{3}{2}}}^{{1\over 3}}
\|\partial_z{\omega\over r}\|_{L^{\frac{3}{2}}}^{{2\over 3}}
\|\partial_z\omega\|_{L^{\frac{3}{2}}}^{\frac{2}{3}}
\Bigl(\|\partial_z\partial_r\omega\|_{L^{\frac{3}{2}}}^{\frac{1}{3}}
+\|\partial_z{\omega\over r}\|_{L^{\frac{3}{2}}}^{{1\over 3}}\Bigr)
\\&
\lesssim
\|{\omega\over r}\|_{L^{\frac{3}{2}}}^{{1\over 3}}
\|\partial_z{\omega\over r}\|_{L^{\frac{3}{2}}}
\|\partial_z\omega\|_{L^{\frac{3}{2}}}^{\frac{2}{3}}
+\|{\omega\over r}\|_{L^{\frac{3}{2}}}^{{1\over 3}}
\|\partial_z{\omega\over r}\|_{L^{\frac{3}{2}}}^{{2\over 3}}
\|\partial_z\omega\|_{L^{\frac{3}{2}}}^{\frac{2}{3}}
\|\partial_z\partial_r\omega\|_{L^{\frac{3}{2}}}^{\frac{1}{3}},
\end{split}
\end{equation*}
and consequently  by Lemma \ref{lem2.2} and H\"{o}lder's inequality, we obtain
\begin{equation*}
\begin{split}
\|\partial_zu^z{\omega\over r}\|_{L^{\frac{3}{2}}}
\|\partial_r\omega\|_{L^{\frac{3}{2}}}^{\frac{1}{2}}
\lesssim &
\|{\omega\over r}\|_{L^{\frac{3}{2}}}^{{1\over 3}}
\|\partial_z{\omega\over r}\|_{L^{\frac{3}{2}}}
\|\partial_z\omega\|_{L^{\frac{3}{2}}}^{\frac{2}{3}}
\|\partial_r\omega\|_{L^{\frac{3}{2}}}^{\frac{1}{2}}
\\&
\quad+
\|{\omega\over r}\|_{L^{\frac{3}{2}}}^{{1\over 3}}
\|\partial_z{\omega\over r}\|_{L^{\frac{3}{2}}}^{{2\over 3}}
\|\partial_z\omega\|_{L^{\frac{3}{2}}}^{\frac{2}{3}}
\|\partial_z|\partial_r\omega|^{3\over4}\|_{L^2}^{\frac{1}{3}}
\|\partial_r\omega\|_{L^{\frac{3}{2}}}^{\frac{7}{12}},
\end{split}
\end{equation*}
which follows
\begin{equation*}
\begin{split}
\|\partial_zu^z{\omega\over r}\|_{L^{\frac{3}{2}}}
\|\partial_r\omega\|_{L^{\frac{3}{2}}}^{\frac{1}{2}}
&\leq
\varepsilon\|\partial_z|\partial_r\omega|^{3\over4}\|_{L^2}^2
+C\|{\omega\over r}\|_{L^{\frac{3}{2}}}^{{1\over 3}}
\|\partial_z{\omega\over r}\|_{L^{\frac{3}{2}}}
\|\partial_z\omega\|_{L^{\frac{3}{2}}}^{\frac{2}{3}}
\|\partial_r\omega\|_{L^{\frac{3}{2}}}^{\frac{1}{2}}
\\&
\qquad+C_{\varepsilon}
\|{\omega\over r}\|_{L^{\frac{3}{2}}}^{{2\over 5}}
\|\partial_z{\omega\over r}\|_{L^{\frac{3}{2}}}^{{4\over 5}}
\|\partial_z\omega\|_{L^{\frac{3}{2}}}^{\frac{4}{5}}
\|\partial_r\omega\|_{L^{\frac{3}{2}}}^{\frac{7}{10}}.
\end{split}
\end{equation*}
On the other hand, thanks to Proposition \ref{prop2.2} again, we find
\begin{equation*}
\begin{split}
 \|g_{1}\|_{L^\frac{3}{2}} \lesssim 
 \|\partial_zu^r\|_{L^6}
\|\partial_z\omega\|_{L^2}+\||\omega|^{1\over2}\|_{L^{6}}
\|\partial_z|\omega|^{3\over2}
\|_{L^2}\lesssim \|\partial_z\omega\|_{L^2}^2+\|\omega\|_{L^{3}}^{{1\over2}}
\|\partial_z|\omega|^{3\over2}\|_{L^2}.
\end{split}
\end{equation*}
Thus in view of \eqref{3.37}, we obtain
\begin{equation*}
\begin{split}
&\frac{d}{dt}\|\partial_r\omega\|_{L^{\frac{3}{2}}}^{\frac{3}{2}}+
\|\partial_z|\partial_r\omega|^{3\over4}\|_{L^2}^2\lesssim
\Bigl(\|{u^r\over r}\|_{L^\infty}+\|u^z\|_{L^\infty}^2\Bigr)
\|\partial_r\omega\|_{L^{\frac{3}{2}}}^{\frac{3}{2}}+\|{\omega\over r}\|_{L^{\frac{3}{2}}}^{{2\over 5}}
\|\partial_z{\omega\over r}\|_{L^{\frac{3}{2}}}^{{4\over 5}}
\|\partial_z\omega\|_{L^{\frac{3}{2}}}^{\frac{4}{5}}
\|\partial_r\omega\|_{L^{\frac{3}{2}}}^{\frac{7}{10}}\\
&
+\Bigl(\|{u^r\over r}\|_{L^\infty}\|{\omega\over r}\|_{L^{\frac{3}{2}}}
+\|\partial_z\omega\|_{L^2}^2+\|\omega\|_{L^{3}}^{{1\over2}}
\|\partial_z|\omega|^{3\over2}\|_{L^2} 
\\&\qquad\qquad\qquad
\qquad+ \|\partial_{z}(\frac{b^2}{r^2})\|_{L^{\frac{3}{2}}}+\|\nabla\frac{b}{r}\|_{L^{2}}\|\partial_{z}\partial_{r}b\|_{L^2}
+\|\partial_{z}\nabla\frac{b}{r}\|_{L^{2}}\|\partial_{r}b\|_{L^2}\Bigr)\|\partial_r\omega\|_{L^{\frac{3}{2}}}^{\frac{1}{2}},
\end{split}
\end{equation*}
from which, together with Lemma \ref{lem2.2} and Proposition \ref{prop2.2}, one can see that
\begin{equation*}
\begin{split}
&\frac{d}{dt}\|\partial_r\omega\|_{L^{\frac{3}{2}}}^{\frac{3}{2}}+
\|\partial_z|\partial_r\omega|^{3\over4}\|_{L^2}^2\lesssim
\Bigl(\|\partial_z{\omega\over r}\|_{L^{\frac{3}{2}, 1}}+\|\omega\|_{L^{3, 1}}^2+\|{\omega\over r}\|_{L^{\frac{3}{2}}}^{{2\over 5}}
\|\partial_z{\omega\over r}\|_{L^{\frac{3}{2}}}^{{4\over 5}}
\|\partial_z\omega\|_{L^{\frac{3}{2}}}^{\frac{4}{5}}\Bigr)
\|\partial_r\omega\|_{L^{\frac{3}{2}}}^{\frac{3}{2}}\\
&
+\Bigl(\|\partial_z{\omega\over r}\|_{L^{\frac{3}{2}, 1}}\|{\omega\over r}\|_{L^{\frac{3}{2}}}+\|{\omega\over r}\|_{L^{\frac{3}{2}}}^{{2\over 5}}
\|\partial_z{\omega\over r}\|_{L^{\frac{3}{2}}}^{{4\over 5}}
\|\partial_z\omega\|_{L^{\frac{3}{2}}}^{\frac{4}{5}}
+\|\partial_z\omega\|_{L^2}^2+\|\omega\|_{L^{3}}^{{1\over2}}
\|\partial_z|\omega|^{3\over2}\|_{L^2} 
\\&\qquad\qquad\qquad
\qquad+ \|\partial_{z}(\frac{b^2}{r^2})\|_{L^{\frac{3}{2}}}+\|\nabla\frac{b}{r}\|_{L^{2}}\|\partial_{z}\partial_{r}b\|_{L^2}
+\|\partial_{z}\nabla\frac{b}{r}\|_{L^{2}}\|\partial_{r}b\|_{L^2}\Bigr)\|\partial_r\omega\|_{L^{\frac{3}{2}}}^{\frac{1}{2}}.
\end{split}
\end{equation*}
Then Gronwall's inequality implies that
\begin{equation}\label{3.38}
\begin{aligned}
&\|\partial_r\omega\|_{L^\infty_t(L^{\frac{3}{2}})}
+
\|\partial_z\partial_r\omega\|_{L^2_t(L^{\frac{3}{2}})}\\
&
\leq
C
\exp\{C\bigl(
\|\partial_z{\omega\over r}\|_{L^1_t(L^{\frac{3}{2}, 1})}+\|\omega\|_{L^2_t(L^{3, 1})}^2
+\|{\omega\over r}\|_{L^2_t(L^{\frac{3}{2}})}^{{2\over 5}}
\|\partial_z{\omega\over r}\|_{L^2_t(L^{\frac{3}{2}})}^{{4\over 5}}
\|\partial_z\omega\|_{L^2_t(L^{\frac{3}{2}})}^{\frac{4}{5}}\bigr)\}
\\&
\times
\Bigl(\|\partial_r\omega_0\|_{L^{\frac{3}{2}}}
+ \|\partial_z{\omega\over r}\|_{L^1_t(L^{\frac{3}{2}, 1})}\|{\omega\over r}\|_{L^\infty_t(L^{\frac{3}{2}})}
+\|{\omega\over r}\|_{L^2_t(L^{\frac{3}{2}})}^{{2\over 5}}
\|\partial_z{\omega\over r}\|_{L^2_t(L^{\frac{3}{2}})}^{{4\over 5}}
\|\partial_z\omega\|_{L^2_t(L^{\frac{3}{2}})}^{\frac{4}{5}}\\
&\qquad+\|\partial_z\omega\|_{L^{2}_t(L^2)}^2+\|\omega\|_{L^\infty_t(L^{3})}^{{1\over2}}
\|\partial_z|\omega|^{3\over2}\|_{L^1_t(L^2)}+\|\partial_{z}(\frac{b^2}{r^2})\|_{L^1_t(L^{\frac{3}{2}})}\\
&\qquad 
+\|\nabla\frac{b}{r}\|_{L^2_t(L^{2})}\|\partial_{z}\partial_{r}b\|_{L^2_t(L^2)}
+\|\partial_{z}\nabla\frac{b}{r}\|_{L^2_t(L^2)}\|\partial_{r}b\|_{L^2_t(L^2)} \Big).
\end{aligned}
\end{equation}
Therefore, inserting \eqref{est-omega-32-1}, \eqref{est-nabla-br-1}, and \eqref {3.49} into \eqref{3.38} implies \eqref{est-romega-1}. We then complete the proof of Proposition \ref{prop3.5}.
\end{proof}

\renewcommand{\theequation}{\thesection.\arabic{equation}}
\setcounter{equation}{0}

\section{Proof of Theorem \ref{thm1.1}}\label{sect-4}

\subsection{Existence part of the proof}
First of all, we note that $\omega_0, \, r^{-1}\omega_0\in L^{\frac{3}{2}, 1}(\mathbb{R}^3)$ ensures that $u_0,\,r^{-1}u_0\in L^{3, 1}(\mathbb{R}^3).$ Let $u_0\in L^{3, 1}(\mathbb{R}^3)$ be an axisymmetric vector field without swirl such that $r^{-1}u_0\in L^{3, 1}(\mathbb{R}^3)$, $\omega_0\in L^{\frac{3}{2}, 1}(\mathbb{R}^3)$, and ${r}^{-1}{\omega_0}\in L^{\frac{3}{2}, 1}(\mathbb{R}^3)$, and assume that the initial axisymmetric data $b_{0}\in L^{3, 2}(\mathbb{R}^3)$ with $r^{-1} b_{0}\in L^{3, 2}(\mathbb{R}^3)$.

Let $J_n$ the operator which localizes in low frequencies defined by 
\begin{equation*}
\widehat{J_n f}(\xi)\eqdefa \chi(2^{-n}\xi)\widehat{f}(\xi) \quad (\forall\,\, n \in \mathbb{Z}),
\end{equation*}
 where $\chi(\xi)$ is a radial and regular function, equal to which to $1$ on a ball around zero, and $\widehat{f}(\xi)$ is the Fourier transform of $f$. Since $(u_0, \,B_0)$ is axisymmetrical with the form \eqref{initial-axi-1}, we know that $(J_nu_0, \, J_nB_0)$ is also axisymmetrical with the form \eqref{initial-axi-1} and also is regular (see for example \cite{AHK}). So, by \cite{WG2022}, there exists a unique regular and global in time axisymmetrical solution $(u^n, \,B^n)$ (with the form \eqref{initial-axi-1}) to the system
 \begin{equation*}
	\begin{cases}
	\partial_{t}u^n+u^n\cdot\nabla u^n-\nu_z\partial_{z}^{2} u^n+\nabla\Pi^n =B^n\cdot\nabla B^n,\\
	\partial_{t}b^n+u^n\cdot\nabla b^n-\mu_z\partial_{z}^{2}B^n= B^n\cdot\nabla u^n,\\
	\dive u^n=\dive B^n= 0,\\
	(u^n,\,b^n)|_{t=0}=(J_nu_{0},\,J_nB_{0}),
	\end{cases}
	\end{equation*}
that is,
 \begin{equation}\label{1.4-n-1}
\begin{cases}
&\partial_{t} \omega^n+\nabla\cdot(\omega^nu^n)-\frac{ (u^n)^{r}}{r}\omega^n-\partial_{z}^{2}\omega^n=-\partial_{z}(\frac{(b^n)^{2}}{r}),\\
&\partial_t \frac{ b^n}{r} +\nabla\cdot(\frac{ b^n}{r} u^n)-\partial_{z}^{2} \frac{ b^n}{r} =0,\\
&u^n=(-\Delta)^{-1}\nabla\times\,(\omega^ne_{\theta}),\\
&(\omega^n,\,b^n)|_{t=0}=(J_n\omega_{0},\,J_nb_{0}).
\end{cases}
\end{equation}
Notice that $J_n\omega_0$ and $\frac{J_n\omega_0}{r}$ are uniformly bounded in
$L^{{3\over 2},1}(\mathbb{R}^3)$, and $J_n b_{0}$ and $\frac{J_n b_{0}}{r}$ are uniformly bounded in $L^{3,2}(\mathbb{R}^3)$, we then obtain from Propositions \ref{prop3.1} and \ref{prop3.2} that:
 \begin{equation}\label{est-unif-bdd-1}
\begin{split}
 &\{(u^n, \omega^n, b^n)\}_{n\in \mathbb{N}} \quad \mbox{ is uniformly bounded  (u.b. for short) in}\\
  &\qquad\qquad\qquad\qquad L^\infty_{loc}(\mathbb{R}^+; \dot{W}^{1, \frac{3}{2}})\times L^\infty_{loc}(\mathbb{R}^+; L^{\frac{3}{2}, 1})\times L^\infty_{loc}(\mathbb{R}^+; L^{3, 2});\\
 &\{(\frac{(u^n)^r}{r}, \frac{b^n}{r})\}_{n\in \mathbb{N}}\quad \mbox{ is u.b. in}\quad L^\infty_{loc}(\mathbb{R}^+; L^{3, 1})\times L^\infty_{loc}(\mathbb{R}^+; L^{3, 2});\\
 &\{(\partial_zu^n, \partial_z\omega^n)\}_{n\in \mathbb{N}}\quad \mbox{ is u.b. in}\quad L^2_{loc}(\mathbb{R}^+; W^{1,\frac{3}{2}}) \times L^2_{loc}(\mathbb{R}^+; L^{\frac{3}{2}, 1});\\
&\{\frac{(u^n)^r}{r}\}_{n\in \mathbb{N}}\quad \mbox{ is u.b. in}\quad L^2_{loc}\bigl(\mathbb{R}^+;\,L^{\infty}\bigr);\\
&\{(\partial_t u^n, \partial_t\frac{ b^n}{r})\}_{n\in \mathbb{N}}\quad \mbox{ is u.b. in} \quad  L^1_{loc}(\mathbb{R}^+; L^{\frac{3}{2}}) \times L^1_{loc}(\mathbb{R}^+; \dot{W}^{-1,\frac{3}{2}}).
   \end{split}
\end{equation}   
By standard compactness arguments and the  Arzela-Ascoli lemma, we can obtain up to a subsequence denoted again by $(u^n, b^n)$, that $(u^n,b^n)$ converges strongly to  $(u, b)$ in $C_{loc}(\mathbb{R}^+; L^2_{loc})\times C_{loc}(\mathbb{R}^+; \dot{H}^{-\frac{1}{2}}_{loc})$. Interpolating with the fact that  $(u_n, \frac{b^n}{r})$ has uniform bound in  \eqref{est-unif-bdd-1}, we found that $u_n\to u$ in $L^2_{loc}(\mathbb{R}^+; H^{\frac14}_{loc}(\mathbb{R}^3))$ and $\frac{b^n}{r}\to \frac{b}{r}$ in $L^2_{loc}(\mathbb{R}^+; L^{\frac32}_{loc}(\mathbb{R}^3))$. This allows to pass to the limit in the nonlinear terms and we conclude that $(\omega^nu^n\to \omega u , \frac{b^n}{r}u^n\to \frac{b}{r}\,u, \frac{b^n}{r}b^n\to \frac{b}{r}b, \frac{ (u^n)^{r}b^n}{r}\to \frac{ u^{r}b}{r})$ in ${\mathcal D}'$. Finally, by passing to the limit in the system \eqref{1.4-n-1} we obtain a global in time,  axisymmetric solution, without swirl,  $(u, B)$ of the system \eqref{1.2}.

\subsection{Uniqueness part of the proof}

In order to prove the uniqueness of the solution for the system \eqref{1.4}, let
$(\omega_1, \,b_1)$ and $(\omega_2,\, b_2)$ be two solutions, and  define
$(\delta\omega, \delta b)\eqdefa (\omega_2 -\omega_1, b_2-b_1)$ their differences, which verifies the following system:
\begin{equation}\label{differ-eqns-1}
  \begin{cases}
&\partial_t\delta\omega+(u_2\cdot \nabla)\delta\omega
-\partial_{z}^2\delta\omega=-( \delta u\cdot \nabla)\omega_1+{u^r_2\over r}\delta\omega
+{ \delta u^r\over r}\omega_1+\partial_{z}[\delta b(\frac{b_{1}+b_{2}}{r})],\\
&\partial_{t}\delta b+(u_{2}\cdot\nabla)\delta b-\partial_{z}^{2}\delta b=-(\delta u\cdot\nabla)b_{1}+\frac{u_{2}^{r}}{r}\delta b+\frac{\delta u^{r}}{r}b_{1},\\
&(\delta\omega,\delta b)_{|t=0}=(0,0).
  \end{cases}
\end{equation} 
The functional framework where we control the differences of the two solutions is  $L^p(\mathbb{R}^3)$ with
${6\over5}\le p<{3\over2}.$
Applying the energy method to the first equation in \eqref{differ-eqns-1} implies
 \begin{equation*}
\begin{split}
\frac{1}{p}\frac{d}{dt}\|\delta\omega\|_{L^{p}}^{p}
+&{4(p-1)\over p^2}
\Big\|\partial_z\vert\delta\omega\vert^{\frac{p}{2}}
\Big\|_{L^2}^2\le
\|{u^r_2\over r}\|_{L^\infty}\|\delta\omega\|_{L^{p}}^{p}
+\|{\omega_1 \delta u^r\over r}\|_{L^{p}}\|\delta\omega\|_{L^{p}}^{p-1}
\\&
+\|( \delta u\cdot \nabla)\omega_1\|_{L^{p}}
\|\delta\omega\|_{L^{p}}^{p-1}+\|\partial_{z}\delta b\|_{L^{\frac{6p}{6-p}}}
\|\frac{b_{1}+b_{2}}{r}\|_{L^{6}}\|\delta\omega\|_{L^{p}}^{p-1}\\
&+\|\delta b\|_{L^{\frac{6p}{6-p}}}\|\partial_{z}\frac{b_{1}+b_{2}}{r}\|_{L^{6}}\|\delta\omega\|_{L^{p}}^{p-1}.
\end{split}
\end{equation*}
Using H\"older's inequality, the Sobolev embedding, Proposition \ref{prop2.2}, and Lemma \ref{lem2.1}, one gets
 \begin{equation*}
\begin{split}
&\|{\omega_1 \delta u^r\over r}\|_{L^{p}}
+\|( \delta u\cdot\nabla)\omega_1\|_{L^{p}}\\
&\le
\bigl(\|{\omega_1\over r}\|_{L^{3\over2}}
+\|\partial_r\omega_1\|_{L^{3\over2}}\bigr)
\| \delta u^r\|_{L^{3p\over3-2p}}
+\|\partial_z\omega_1\|_{L^6_h(L^{\frac{3}{2}}_v)}
\|\delta u^z\|_{L^{6p\over6-p}_h(L^{\frac{3p}{3-2p}}_v)}
\\&
\lesssim
\bigl(\|{\omega_1\over r}\|_{L^{3\over2}}
+\|\partial_r\omega_1\|_{L^{3\over2}}\bigr)
\|\partial_z\delta\omega\|_{L^{p}}
+\|\partial_z\partial_r\omega_1\|_{L^{\frac{3}{2}}}
\| \delta u^z\|_{L^{6p\over6-p}_h(L^{\frac{3p}{3-2p}}_v)}
\\&
\lesssim
\Bigl(\|{\omega_1\over r}\|_{L^{3\over2}}
+\|\partial_r\omega_1\|_{L^{3\over2}}\Bigr)
\Big\|\partial_z\vert\delta\omega\vert^{p\over2}\Big\|_{L^2}
\Vert\delta\omega\Vert_{L^{p}}^{2-p\over2}
+\|\partial_z\partial_r\omega_1\|_{L^{\frac{3}{2}}}
\|\delta u^z\|_{L^{6p\over6-p}_h(L^{\frac{3p}{3-2p}}_v)}.
\end{split}
\end{equation*}
Concerning  $\| \delta u^z\|_{L^{6p\over6-p}_h(L^{\frac{3p}{3-2p}}_v)}$
and using the identity $\Delta \delta u^z
=\partial_r\delta\omega+r^{-1}\delta\omega$, we obtain by integration by parts that
$
| \delta u^z|
\lesssim
\frac{1}{|\cdot|^2}\star|\delta\omega|.
$
Then, using the convolution law, we get
 \begin{equation*}
\begin{split}
\| \delta u^z\|_{L^{6p\over6-p}_h(L^{\frac{3p}{3-2p}}_v)}
\lesssim
\|\delta\omega\|_{L^{{p},{6p\over6-p}}_h
(L^p_v)}
\lesssim
\|\delta\omega\|_{L^p}.
\end{split}
\end{equation*}
The Young inequality implies that
\begin{equation*}
\begin{split}
&\frac{d}{dt}\|\delta\omega\|_{L^{p}}^{p}+
{2(p-1)\over p}\Big\|\partial_z\vert\delta\omega\vert^{\frac{p}{2}}
\Big\|_{L^2}^2
\lesssim
\Big(\|{u^r_2\over r}\|_{L^\infty}+\|{\omega_1\over r}\|_{L^{3\over2}}^2
+\|\partial_r\omega_1\|_{L^{3\over2}}^2
+\|\partial_z\partial_r\omega_1
\|_{L^{\frac{3}{2}}}\Big)
\|\delta\omega\|_{L^{p}}^{p}\notag\\
&\qquad+\Big(\|\partial_z\delta b\|_{L^{\frac{6p}{6-p}}} \|(\frac{b_{1}}{r},\,\frac{b_{2}}{r})\|_{L^{6}}+\|\delta b\|_{L^{\frac{6p}{6-p}}}
 \|\partial_{z}(\nabla\frac{b_{1}}{r},\,\nabla\frac{b_{2}}{r})\|_{L^{2}}\Big)\|\delta\omega\|_{L^{p}}^{p-1},
\end{split}
\end{equation*}
 which follows that
 \begin{equation*}
\begin{split}
&\|\delta\omega\|_{L^\infty_t(L^{p})}^{p}+
\Big\|\partial_z\vert\delta\omega\vert^{\frac{p}{2}}
\Big\|_{L^2_t(L^2)}^2
\leq C\mathcal{F}_1(t)\|\delta\omega\|_{L^\infty_t(L^{p})}^{p}\\
&+C\Big(\|\partial_z\delta b\|_{L^2_t(L^{\frac{6p}{6-p}})} \|(\frac{b_{1}}{r},\,\frac{b_{2}}{r})\|_{L^2_t(L^{6})}+\|\delta b\|_{L^\infty_t(L^{\frac{6p}{6-p}})}
 \|\partial_{z}(\nabla\frac{b_{1}}{r},\,\nabla\frac{b_{2}}{r})\|_{L^1_t(L^{2})}\Big)\|\delta\omega\|_{L^\infty_t(L^{p})}^{p-1},
\end{split}
\end{equation*}
where
\begin{equation}\label{def-f1}
\begin{split}
&\mathcal{F}_1(t)\eqdefa \|{u^r_2\over r}\|_{L^1_t(L^\infty)}+\|{\omega_1\over r}\|_{L^2_t(L^{3\over2})}^2
+\|\partial_r\omega_1\|_{L^2_t(L^{3\over2})}^2
+\|\partial_z\partial_r\omega_1
\|_{L^1_t(L^{\frac{3}{2}})}.
\end{split}
\end{equation}
Hence, due to the fact that $\mathcal{F}_1(t) \rightarrow 0$ as $t\rightarrow 0^+$, we obtain that, there is $\epsilon_1>0$ so small that, if $t\in [0, \epsilon_1]$, there holds
\begin{equation}\label{4.21-34d}
  \begin{split}
&\|\delta\omega\|_{L^\infty_t(L^p)}+\Big\|\partial_{z}|\delta \omega|^{\frac{p}{2}}\Big\|_{L^2_t(L^{2})}^{\frac{2}{p}} \leq C\mathcal{F}_2(t)\|\partial_z\delta b\|_{L^2_t(L^{\frac{6p}{6-p}})}+C \mathcal{F}_3(t) \|\delta b\|_{L^\infty_t(L^{\frac{6p}{6-p}})},
  \end{split}
  \end{equation}
  where
\begin{equation}\label{def-f23}
  \begin{split}
&\mathcal{F}_2(t)\eqdefa \|(\frac{b_{1}}{r},\,\frac{b_{2}}{r})\|_{L^2_t(L^{6})},\quad\mathcal{F}_3(t)\eqdefa \|\partial_{z}(\nabla\frac{b_{1}}{r},\,\nabla\frac{b_{2}}{r})\|_{L^1_t(L^{2})}.
  \end{split}
  \end{equation}
  Due to \eqref{psmall-1}, we arrive at
  \begin{equation}\label{4.21-34e}
  \begin{split}
&\|\delta\omega\|_{L^\infty_t(L^p)}+\|\partial_{z}\delta \omega\|_{L^2_t(L^{p})}\\
&\leq C\mathcal{F}_2(t)\|\partial_z|\delta b|^{\frac{3p}{6-p}}\|_{L^2_t(L^{2})}\|\delta b\|_{L^\infty_t(L^{\frac{6p}{6-p}})}^{1-\frac{3p}{6-p}}+C \mathcal{F}_3(t) \|\delta b\|_{L^\infty_t(L^{\frac{6p}{6-p}})}.
  \end{split}
  \end{equation}
On the other hand, from the $\delta b$ equation of \eqref{differ-eqns-1}, the energy estimates imply
\begin{align}\label{4.3}
\frac{\mathrm{d}}{\mathrm{d}t}\|\delta b\|_{L^{\frac{6p}{6-p}}}^{\frac{6p}{6-p}}+&
\Big\|\partial_{z}|\delta b|^{\frac{3p}{6-p}}\Big\|_{L^{2}}^{2}
\notag\\&
\lesssim
\bigl(\|(\delta u\cdot\nabla)b_{1}\|_{L^{\frac{6p}{6-p}}}+\|\frac{u_{2}^{r}}{r}\|_{L^{\infty}}\|\delta b\|_{L^{\frac{6p}{6-p}}}+\|\frac{\delta u^{r}}{r}b_{1}\|_{L^{\frac{6p}{6-p}}}\bigr)
\|\delta b\|_{L^{\frac{6p}{6-p}}}^{\frac{7p-6}{6-p}}.
\end{align}
Again, using H\"older inequality, Proposition \ref{prop2.2}, Sobolev
embedding, we have
\begin{align*}
&
\|(\delta u\cdot\nabla)b_{1}\|_{L^{\frac{6p}{6-p}}}
\lesssim
\|\partial_z\delta \omega\|_{L^p}\|\partial_rb_{1}\|_{L^2}+\|\delta \omega\|_{L^p}\|\partial_z\nabla\,b_{1}\|_{L^2},
\\&
\|\frac{\delta u^{r}}{r}b_{1}\|_{L^{\frac{6p}{6-p}}}
\lesssim
\|\delta u^{r}\|_{L^{\frac{3p}{3-2p}}}\|\frac{b_1}{r}\|_{L^2}
\lesssim
\|\partial_z\delta\omega\|_{L^{p}}\|\frac{b_1}{r}\|_{L^2}.
\end{align*}
Notice that
\begin{equation}\label{u-b1-1-1}
  \begin{split}
\|(\delta u\cdot\nabla)b_{1}\|_{L^{\frac{6p}{6-p}}}&\lesssim\|\delta u^r\partial_rb_{1}\|_{L^{\frac{6p}{6-p}}}+\|\delta u^z\partial_zb_{1}\|_{L^{\frac{6p}{6-p}}}\\
&\lesssim\|\delta u^r\|_{L^{\frac{3p}{3-2p}}}\|\partial_rb_{1}\|_{L^2}+\|\delta u^z\|_{L^{\frac{3p}{3-p}}}\|\partial_zb_{1}\|_{L^6}.
  \end{split}
\end{equation}
Form Proposition \ref{prop2.2}, we known that, for ${6\over5}\le p<{3\over2}$, 
\begin{equation*}\|\delta u^r\|_{L^{\frac{3p}{3-2p}}}\lesssim\|\partial_z\delta \omega\|_{L^p}, \quad \|\delta u^{z}\|_{L^{\frac{3p}{3-p}}}\lesssim\|\delta\omega\|_{L^{p}},
\end{equation*}
which along with \eqref{u-b1-1-1} implies
\begin{equation*}
  \begin{split}
\|(\delta u\cdot\nabla)b_{1}\|_{L^{\frac{6p}{6-p}}}\lesssim \|\partial_z\delta \omega\|_{L^p}\|\partial_rb_{1}\|_{L^2}+\|\delta \omega\|_{L^p}\|\partial_z\nabla\,b_{1}\|_{L^2}.
  \end{split}
\end{equation*}
For the last term of \eqref{4.3}, again using the above embedding inequality, we obtain
$$\|\frac{\delta u^{r}}{r}b_{1}\|_{L^{\frac{6p}{6-p}}}\lesssim\|\delta u^{r}\|_{L^{\frac{3p}{3-2p}}}\|\frac{b_1}{r}\|_{L^2}\lesssim
\|\partial_z\delta\omega\|_{L^{p}}\|\frac{b_1}{r}\|_{L^2}$$
and then, we arrive at
\begin{equation*}
  \begin{split}
\frac{\mathrm{d}}{\mathrm{d}t}\|\delta b\|_{L^{\frac{6p}{6-p}}}^{\frac{6p}{6-p}}+
\Big\|\partial_{z}|\delta b|^{\frac{3p}{6-p}}\Big\|_{L^{2}}^{2}\lesssim
\bigl(&\|\partial_z\delta \omega\|_{L^p}\|\partial_rb_{1}\|_{L^2}+\|\delta \omega\|_{L^p}\|\partial_z\nabla\,b_{1}\|_{L^2}
\\&
+\|\frac{u_{2}^{r}}{r}\|_{L^{\infty}}\|\delta b\|_{L^{\frac{6p}{6-p}}}+\|\partial_z\delta\omega\|_{L^{p}}\|\frac{b_1}{r}\|_{L^2}\bigr)
\|\delta b\|_{L^{\frac{6p}{6-p}}}^{\frac{7p-6}{6-p}}.
  \end{split}
  \end{equation*}
Hence, we have 
  \begin{equation}\label{4.2-22}
  \begin{split}
&\|\delta b\|_{L^\infty_t(L^{\frac{6p}{6-p}})}^{\frac{6p}{6-p}}+
\Big\|\partial_{z}|\delta b|^{\frac{3p}{6-p}}\Big\|_{L^2_t(L^{2})}^{2} \leq C\|{r}^{-1}{u_{2}^{r}}\|_{L^1_t(L^{\infty})}\|\delta b\|_{L^\infty_t(L^{\frac{6p}{6-p}})}^{\frac{6p}{6-p}}\\
&\quad+
 C(\|(\partial_rb_{1}, \frac{b_1}{r})\|_{L^2_t(L^2)}+\|\partial_z\nabla\,b_{1}\|_{L^1_t(L^2)})(\|\partial_z\delta \omega\|_{L^2_t(L^p)}+\|\delta \omega\|_{L^\infty_t(L^p)})\|\delta b\|_{L^\infty_t(L^{\frac{6p}{6-p}})}^{\frac{7p-6}{6-p}}.
  \end{split}
  \end{equation}
 Substituting  \eqref{4.21-34e} into \eqref{4.2-22} gives rise to
   \begin{equation*}
  \begin{split}
&\|\delta b\|_{L^\infty_t(L^{\frac{6p}{6-p}})}^{\frac{6p}{6-p}}+
\Big\|\partial_{z}|\delta b|^{\frac{3p}{6-p}}\Big\|_{L^2_t(L^{2})}^{2}\leq C\|{r}^{-1}{u_{2}^{r}}\|_{L^1_t(L^{\infty})}\|\delta b\|_{L^\infty_t(L^{\frac{6p}{6-p}})}^{\frac{6p}{6-p}}\\
&\quad+
 C\mathcal{F}_4(t) \bigg( \mathcal{F}_2(t)\|\partial_z|\delta b|^{\frac{3p}{6-p}}\|_{L^2_t(L^{2})}\|\delta b\|_{L^\infty_t(L^{\frac{6p}{6-p}})}^{1-\frac{3p}{6-p}}+\mathcal{F}_3(t) \|\delta b\|_{L^\infty_t(L^{\frac{6p}{6-p}})}\bigg)\|\delta b\|_{L^\infty_t(L^{\frac{6p}{6-p}})}^{\frac{7p-6}{6-p}}
  \end{split}
  \end{equation*}
  with
\begin{equation}\label{def-f4}
  \begin{split}
&\mathcal{F}_4(t)\eqdefa \|(\partial_rb_{1}, \frac{b_1}{r})\|_{L^2_t(L^2)}+\|\partial_z\nabla\,b_{1}\|_{L^1_t(L^2)},
  \end{split}
  \end{equation}
  which along with Young's inequality implies  
  \begin{equation*}
  \begin{split}
&\|\delta b\|_{L^\infty_t(L^{\frac{6p}{6-p}})}^{\frac{6p}{6-p}}+
\Big\|\partial_{z}|\delta b|^{\frac{3p}{6-p}}\Big\|_{L^2_t(L^{2})}^{2} \leq C \mathcal{F}_5(t)\|\delta b\|_{L^\infty_t(L^{\frac{6p}{6-p}})}^{\frac{6p}{6-p}}  \end{split}
  \end{equation*}
 with
\begin{equation}\label{def-f5}
  \begin{split}
&\mathcal{F}_5(t)\eqdefa \|{r}^{-1}{u_{2}^{r}}\|_{L^1_t(L^{\infty})}+\mathcal{F}_3(t)\mathcal{F}_4(t)+(\mathcal{F}_2(t)\mathcal{F}_4(t))^2.
  \end{split}
  \end{equation}
Notice that $\mathcal{F}_5(t)\rightarrow 0$ as $t\rightarrow 0+$, so we get that, there exists $\epsilon_0 \in (0, \epsilon_1)$ so small that, if $t\in [0, \epsilon_0]$, there holds
       \begin{equation*} 
  \begin{split}
&\|\delta b\|_{L^\infty_t(L^{\frac{6p}{6-p}})}^{\frac{6p}{6-p}}+
\Big\|\partial_{z}|\delta b|^{\frac{3p}{6-p}}\Big\|_{L^2_t(L^{2})}^{2} =0,
 \end{split}
  \end{equation*}
  which immediately follows from \eqref{4.21-34d} and \eqref{psmall-1} that
  \begin{equation*} 
  \begin{split}
  \|\delta\omega\|_{L^\infty_t(L^p)}=0.
 \end{split}
  \end{equation*}
  Therefore, we obtain $\delta\,b(t)=\delta\omega(t)\equiv 0$ for any $t\in [0, \epsilon_0]$.
 The uniqueness of such strong solutions on the whole time interval $[0, +\infty)$ then follows by a bootstrap argument.

Moreover, the continuity with respect to the initial data may also be obtained by the same argument in the proof of the uniqueness, which ends the proof of Theorem \ref{thm1.1}.

\bigbreak \noindent {\bf Acknowledgments.} G. Gui's research is supported in part by the National Natural Science Foundation of China under Grants 12371211 and 12126359.

\end{document}